\documentclass{amsart}

\usepackage{amsmath,amssymb,amsfonts,amsthm}

\usepackage{mathrsfs}

\usepackage[OT2,OT1]{fontenc}
\newcommand\cyr{%
\renewcommand\rmdefault{wncyr}%
\renewcommand\sfdefault{wncyss}%
\renewcommand\encodingdefault{OT2}%
\normalfont
\selectfont}
\DeclareTextFontCommand{\textcyr}{\cyr}

\usepackage[all,2cell]{xy}
\UseTwocells
\usepackage{color}
\usepackage{tikz}
\usetikzlibrary{matrix}

\xyoption{arc}

\usepackage{hyperref}%To get clickable links to papers

\usepackage{multirow}
\usepackage{rotating}
\usepackage{fullpage}

\renewcommand{\(}{\begin{equation}}
\renewcommand{\)}{\end{equation}}
\newcommand{\bea}{\begin{eqnarray}}
\newcommand{\eea}{\end{eqnarray}}

\theoremstyle{plain}
\newtheorem{theorem}{Theorem}[section]

\newtheorem{proposition}[theorem]{Proposition}

\theoremstyle{definition}
\newtheorem{example}[theorem]{Example}

\theoremstyle{theorem}
\newtheorem{definition}[theorem]{Definition}

\newtheorem{lemmaapp}{Lemma A.\!\!}

\theoremstyle{remark}
\newtheorem{remark}[theorem]{Remark}

\title{Central extensions of mapping class groups from characteristic classes}
\author{Domenico Fiorenza}
\address{Dipartimento di Matematica,``La Sapienza'' Universit\`a di Roma, P.le Aldo Moro, 5, Roma, Italy}
\author{Urs Schreiber}
\address{Charles University Institute of Mathematics, Sokolovsk\'a 83, 186 75 Praha 8, Czech Republic}
\author{Alessandro Valentino}
\address{Max Planck Institut f\"ur Mathematik, Vivatsgasse 7, 53113 Bonn, Germany}

\begin{document}
%%

%\vspace{-.5cm}
%
%
%\hspace{.5cm}
\begin{abstract}
We characterize, %in generality
for every higher smooth stack equipped
with ``tangential structure'', the induced higher group extension of the
geometric realization of its higher automorphism stack. We show that when
restricted to smooth manifolds equipped with higher degree topological
structures, this produces higher extensions of homotopy types of diffeomorphism groups.
Passing to the groups of connected components, we obtain abelian
extensions of mapping class groups and we derive sufficient conditions
for these being central.  We show as a special case that this
provides an elegant re-construction of Segal's approach to $\mathbb{Z}$-extensions of mapping
class groups of surfaces that provides the anomaly cancellation of the modular functor in Chern-Simons
theory.  Our construction generalizes
Segal's approach to higher central extensions of mapping class
groups of higher dimensional manifolds with higher tangential
structures, expected to provide the analogous anomaly
cancellation for higher dimensional TQFTs.\end{abstract}
\maketitle
\begin{flushright}
\emph{``Everything in its right place''}\\
Kid A, Radiohead
\end{flushright}
\setcounter{tocdepth}{1}
\tableofcontents
%%\newpage

\section{Introduction}

In higher (stacky) geometry, there is a general and fundamental class of higher (stacky) group extensions: for $\psi : Y \to B$ any morphism between higher stacks, the automorphism group stack of $Y$ over $B$
extends the automorphisms of $Y$ itself by the loop object of the mapping stack $[Y,B]$ based at $\psi$.
{
Schematically this extension is of the following form
$$
  \left\{
  \;\;
  \raisebox{20pt}{
  \xymatrix{
    Y
    \ar@/^1.4pc/[d]|\psi_<<{\ }="s"
    \ar@/_1.4pc/[d]|\psi^>>{\ }="t"
    \\
    B
    \ar@{=>} "s"; "t"
  }
  }
  \;\;
  \right\}
  \;\;
  \longrightarrow
  \;\;
  \left\{
  \raisebox{20pt}{
  \xymatrix{
    Y \ar[dr]|\psi^{\ }="t" \ar[rr]^-{\simeq}_{\ }="s"
     &&
    Y \ar[dl]|\psi
    \\
    & B
    \ar@{=>} "s"; "t"
  }
  }
  \right\}
  \;\;
  \longrightarrow
  \;\;
  \left\{
      Y  \xrightarrow{\phantom{m}\simeq\phantom{m}}
      Y
  \right\}
$$
but the point is that all three items here are themselves realized ``internally'' as higher group stacks.
}
This is not hard to prove \cite[prop. 3.6.16]{schreiber}, but as a general abstract fact it has many non-trivial incarnations.
{ Here we are concerned with a class of examples of these extensions for the case of smooth higher stacks,
i.e. higher stacks over the site of all smooth manifolds.}

In \cite{hgp} it was shown that for the choice that $B=\mathbf{B}^nU(1)_{\mathrm{conn}}$ is the universal moduli stack for { degree $n+1$} ordinary differential cohomology, then these extensions reproduce and generalize the Heisenberg-Kirillov-Kostant-Souriau-extension
from prequantum line bundles to higher ``prequantum gerbes'' which appear in the local (or ``extended'') geometric quantization
of higher dimensional field theories.

Here we consider a class of examples at the other extreme: we consider the
case in which { $Y$ is a smooth manifold (regarded as the stack that it presents),
but $B$ is geometrically discrete (i.e., it is a locally constant  $\infty$-stack),
and particularly the
case that $B$ is the homotopy type of the classifying space of the general linear group. }
{ This means that the slice automorphism group (the middle term above) becomes a smooth group stack
that extends the smooth diffeomorphism group of $Y$ (the item on the right above) by a locally constant higher group stack
(the item on the left)}.

{ We are interested in the homotopy type of this higher stacky extension of the diffeomorphism group,
that is in the geometric realization of the smooth slice group stack.}
{ In general, geometric realization of higher smooth group stacks will not preserve the above
extension, but here it does, due to the fact that $B$ is assumed to be geometrically discrete.
This resulting class of extensions is our main Theorem \ref{prop.extension} below. It uses that
geometric realization of smooth $\infty$-stacks happens to preserve homotopy
fibers over geometrically discrete objects \cite[thm. 3.8.19]{schreiber}.
% the general extension theorem essentially passes along geometric realization.
Hence, where the internal extension theorem gives extensions of smooth diffeomorphism groups
by higher homotopy types, after geometric realization we obtain higher extensions of the homotopy type of
diffeomorphism groups, and in particular of mapping class groups.}

{ We emphasize that it is the interplay between smooth higher stacks and their geometric realization that
makes this work: one does not see diffeomorphism groups, nor their homotopy types, when forming the above
extension in the plain homotopy theory of topological spaces. So, even though the group extensions that we
study are geometrically discrete, they encode information about smooth diffeomorphism groups.}

A key application where extensions of the mapping class group traditionally play a role is anomaly cancellation
in 3-dimensional topological field theories, e.g., in 3d Chern-Simons theory, see, e.g., \cite{witten}.

Our general extension result reduces to a new and elegant construction of the anomaly cancellation
construction for modular functors in 3d Chern-Simons theory, and
naturally generalizes this to higher extensions relevant for higher dimensional topological quantum field theories (TQFTs).

More in detail, by functoriality, a 3d TQFT associates to any connected oriented surface $\Sigma$ a vector space $V_\Sigma$ which is
a linear representation of the oriented mapping class group $\Gamma^{\mathrm{or}}(\Sigma)$ of $\Sigma$. However, if the 3d theory has an ``anomaly'', then the vector space $V_\Sigma$ fails to be a genuine representation of $\Gamma^{\mathrm{or}}(\Sigma)$, and it rather is only a projective representation. One way to think of this phenomenon is to look at anomalous theories as relative theories, that intertwine between the trivial theory and an invertible theory, namely the anomaly. See, e.g. \cite{freed-teleman,fiorenza-valentino}. In particular, for an anomalous TQFT of the type obtained from modular tensor categories with nontrivial central charge \cite{turaev,bakalov-kirillov}, the vector space $V_\Sigma$ can be naturally realised as a genuine representation of a $\mathbb{Z}$-central extension
 \begin{equation}
 0\to \mathbb{Z}\to\widehat{\Gamma(\Sigma)}\to \Gamma(\Sigma)\to 1
 \end{equation}
of the mapping class group $\Gamma(\Sigma)$. As suggested in Segal's celebrated paper on conformal field theory \cite{SegalCFT}, these data admit an interpretation as a genuine functor where one replaces 2-dimensional and 3-dimensional manifolds by suitable ``enriched'' counterparts, in such a way that the automorphism group of an enriched connected surface is the relevant $\mathbb{Z}$-central extension of the mapping class group of the underlying surface. Moreover, the set of (equivalence classes of) extensions of a 3-manifold with prescribed (connected) boundary behaviour is naturally a $\mathbb{Z}$-torsor. In \cite{SegalCFT} the extension consists in a ``rigging'' of the 3-manifold, a solution which is not particularly simple, and which is actually quite ad hoc for the 3-dimensional case. Namely, riggings are based on the contractibility of Teichm\"uller spaces, and depend on the properties of the $\eta$-invariant for Riemannian metrics on 3-manifolds with boundary. On the other hand, in \cite{SegalCFT} it is suggested  that simpler variants of this construction should exist, the \emph{leitmotiv} being that of associating functorially to any connected surface a space with fundamental group $\mathbb{Z}$. Indeed, there is a well known realization of extended surfaces as surfaces equipped with a choice of a Lagrangian subspace in their first real cohomology group. This is the point of view adopted, e.g., in \cite{bakalov-kirillov}. The main problem with this approach is the question of how to define a corresponding notion for an extended 3-manifold.

In the present work we show how a natural way of defining enrichments of 2-and-3-manifolds, which are topological (or better homotopical) in nature, and in particular do not rely on special features of the dimensions 2 and 3. Moreover, they have the advantage of being immediately adapted to a general TQFT framework. Namely, we consider enriched manifolds as $(X,\xi)$-framed manifolds in the sense of \cite{lurieTFT}. In this way, we in particular recover the fact that the simple and natural notion of $p_1$-structure, i.e. a trivialization of the first Pontryagin class, provides a very simple realization of Segal's prescription by showing how it naturally drops out as a special case of the ``higher modularity'' encoded in the $(\infty,n)$-category of framed cobordisms. { This is { discussed in detail} in Section \ref{sect.p1structures} below.}\\
Finally, if one is interested in higher dimensional Chern-Simons theories, the notable next case being
7-dimensional Chern-Simons theory \cite{7dCS}, then the above discussion gives general means for determining
and constructing the relevant higher extensions of diffeomorphism groups of higher dimensional manifolds.\\
More on this is going to be discussed elsewhere.\\

The present paper is organised as follows.
In section \ref{section.framed} we discuss the ambient homotopy theory $\mathbf{H}^{\infty}$ of smooth higher stacks,
and we discuss how smooth manifolds and homotopy actions of $\infty$-groups can be naturally regarded as objects in its slice $\infty$-category over the homotopy type $\mathscr{B}GL(n;\mathbb{R})$ of the mapping stack $\mathbf{B}GL(n;\mathbb{R})$ of principal $GL(n;\mathbb{R})$-bundles. \\
In section \ref{rhoframed} we introduce the notion of a $\rho$-framing (or $\rho$-structure) over a smooth manifold, and study extensions of their automorphism $\infty$-group. We postpone the proof of the extension result to the Appendix.\\
In section \ref{section.lifting} we discuss the particular but important case of $\rho$-structures arising from the homotopy fibers of morphisms of $\infty$-stacks, which leads to Theorem \ref{prop.extension}, the main result of the present paper. In this section we also consider the case of manifolds with boundaries.\\
In section \ref{section.mapping}, we apply the abstract machinery developed in the previous sections to the concrete case of the mapping class group usually encountered in relation to topological quantum field theories.\\
The Appendix contains a proof of the extension result in section \ref{section.lifting}.\\

Throughout, we freely use the language of $\infty$-categories, as
developed in \cite{lurieHTT}. There are various equivalent models for these, such
as by simplicially enriched categories as well as by quasi-categories, but
since these are equivalent, we mostly do not specify the model, and
the reader is free to think of whichever model they prefer. \\

\noindent\textbf{Acknowledgements.} The authors would like to thank Oscar Randal-Williams and Chris Schommer-Pries for useful discussions.
%
%{Here we should say something more about Segal's original construction, the definition of extended surfaces as surfaces with the choice of a Lagrangian subspace of their real $H^1$, the fact that this gives $\mathbb{Z}$-central extensions of the mapping class group, the fact that Segal says that the really crucial thing are not Lagrangian subspaces but is to associate functorially to a surface $M$ a space $X_M$ with $\pi_1 X_M=\mathbb{Z}$. Also we should say that in the Lagrangian subspace perspective it is not clear what one should take as an extended 3-manifold, Segal says something about rigged 3-manifolds, but the whole thing appears to be more complicated than it should. Here we show how the simple notion of $p_1$-framing does all the job and works equally well for any oriented $k$-manifold with $k\leq 3$. Which is good in a fully extended 3d TQFT perspective. Also, and much needed, we have to say here or in the Introduction that a version of this note, dealing with more general higher extensions and with fully detailed and rigorous proofs is in preparation.}

%\medskip

%The following note is informal in that it does not fully spell out all the proofs.

\section{Framed manifolds}\label{section.framed}

\subsection{From framed cobordism to $(X,\xi)$-manifolds}
The principal player in the celebrated constructions of \cite{GalatiusMadsenTillmannWeiss06, GalatiusRandallWilliams10, lurieTFT} are manifolds with exotic ``tangential structure'' or ``framing''. These framings come in various flavours, from literal $n$-framings, i.e., trivialisations of the (stabilized) tangent bundle to more general and exotic framings called $(X,\xi)$-structures in \cite{lurieTFT}.
%Presumably to keep the note at the lowest possible technical level, Lurie avoids to say explicitly that he is working in a slice. However, this is what he %is secretely doing, and the
Here we make explicit that these structures are most naturally understood in the
slice of a suitable smooth $\infty$-topos over $\mathbf{H}^\infty$ over $\mathscr{B}GL(n;\mathbb{R})$
%is the unifying principle governing all the framings in \cite{lurieTFT}.
%Here we make the role played by $\mathscr{B}GL(n;\mathbb{R})$ more explicit.
This will allow us not only to see Lurie's framings from a unified
perspective, but also to consider apparently more exotic (but actually completely natural) framings given by characteristic classes for the orthogonal group.
\subsubsection{Homotopies, homotopies, homotopies everywhere}
{
The $\infty$-topos of $\infty$-stacks over the site of all smooth manifolds, or equivalently just over the site of
Cartesian spaces among these, we denote by
$$
  \mathbf{H} := \mathrm{Sh}_\infty(\mathrm{SmthMfd}) \simeq \mathrm{Sh}_\infty(\mathrm{CartSp})
$$
\cite[def. 3.1.4]{fiorenza-schreiber-stasheff}.

This is a \emph{cohesive $\infty$-topos} \cite[prop. 4.4.8]{schreiber}, which in particular means (\cite[def. 3.4.1]{schreiber} following \cite{Lawvere}) that the
locally constant $\infty$-stack functor $\mathrm{LConst} : \infty\mathrm{Grp} \longrightarrow \mathbf{H}$
is fully faithful and
has a left adjoint ${\vert -\vert}$ that preserves products
$$
  ( {\vert-\vert} \,\dashv\, \mathrm{LConst} \,\dashv\, \Gamma )
    \;:\;
  \xymatrix{
    \mathbf{H}
      \ar@{->}@<+8pt>[rr]|<\times^-{\vert -\vert}
      \ar@{<-^{)}}@<+0pt>[rr]|{\mathrm{LConst}}
      \ar@<-8pt>[rr]_{\Gamma}
    &&
    \infty\mathrm{Grpd}
  }
  \,.
$$
This extra left adjoint ${\vert - \vert}$ is the operation of sending a smooth $\infty$-stack to its
topological geometric realization, thought of as an $\infty$-groupoid \cite[cor. 4.4.28]{schreiber},\cite[thm. 1.1]{Carchedi15}.
In particular a smooth manifold is sent to its homotopy type.

Notice that the hom-$\infty$-groupoids of any $\infty$-topos $\mathbf{H}$ may be expressed in terms of the internal hom (mapping $\infty$-stack)
construction $[-,-]$ as
$$
  \mathbf{H}(\Sigma_1,\Sigma_2) \simeq \Gamma( [\Sigma_1, \Sigma_2] )
  \,.
$$
But now since the left adjoint ${\vert-\vert}$ exists and preserves products, this means that there naturally exists an alternative $\infty$-category,
which we denote by $\mathbf{H}^\infty$,
with the same objects as $\mathbf{H}$, but with hom-$\infty$-groupoids defined by%
\footnote{
This construction of an $\infty$-category $\mathbf{H}^\infty$ from a cohesive $\infty$-topos
$\mathbf{H}$ is the direct $\infty$-category theoretic analog of what for cohesive 1-categories is called
their ``canonical extensive quality'' in \cite[thm. 1]{Lawvere}.
}
\begin{equation}\label{homotopyenrichment}
  \mathbf{H}^\infty(\Sigma_1,\Sigma_2):=\big\vert[\Sigma_1,\Sigma_2]\big\vert
\end{equation}
Accordingly, we write $\mathbf{Aut}^\infty(\Sigma)$ for the sub-$\infty$-groupod on the invertible elements in $\mathbf{H}^\infty(\Sigma,\Sigma)$.

\medskip

%{ Suggest to remove the following sentence: {\color{green} We will not go into the technicalities of higher toposes or higher smooth stacks in the present work: at any point where one might be unsure on what is precisely going on, mumbling several times the mantra ``$\mathbf{B}G$ is a smooth stack'' will make everything appear suddenly clear. The reader who is skeptical of the effectiveness of these transcendental methods will find a complete and fully rigorous treatment of the theory of higher smooth stacks in \cite{schreiber}.}}
%Also the first sections of \cite{fiorenza-schreiber-stasheff} can serve as a friendly introduction to the subject. Also, a rigorous construction of $\mathbf{H}^\infty$ is beyond the aims of this note, and will be {hopefully} presented in detail elsewhere: here, we will content ourself with an informal description, which will suffice to motivate and justify the construction.
The reason we pass to $\mathbf{H}^\infty$ is that $\mathbf{H}$ itself is too rigid (or, in other words, the homotopy type of its hom-spaces is too simple) for our aims. For instance, given two smooth manifolds $\Sigma_1$ and $\Sigma_2$, the $\infty$-groupoid $\mathbf{H}(\Sigma_1,\Sigma_2)$ is 0-truncated, i.e., it is just a set. Namely,  $\mathbf{H}(\Sigma_1,\Sigma_2)$ is the set of smooth maps from $\Sigma_1$ and $\Sigma_2$ and there are no nontrivial morphisms between smooth maps in $\mathbf{H}(\Sigma_1,\Sigma_2)$. In other words, two smooth maps between  $\Sigma_1$ and $\Sigma_2$ either are equal or they are different:
in this hom-space there's no such thing as ``a smooth map can be smoothly deformed into another smooth map'',
which however is a kind of relation that geometry naturally suggests. To take it into account, we make the topology (or, even better, the smooth structure) of $\Sigma_1$ and $\Sigma_2$ come into play, and we use it to informally define $\mathbf{H}^\infty(\Sigma_1,\Sigma_2)$ as the $\infty$-groupoid whose objects are smooth maps between $\Sigma_1$ and $\Sigma_2$, much as for $\mathbf{H}(\Sigma_1,\Sigma_2)$, but whose 1-morphism are the smooth homotopies between smooth maps, and we also have 2-morphisms given by homotopies between homotopies, 3-morphisms given by homotopies  between homotopies between homotopies, and so on. {

%{ The} formal definition
%is\footnote{{This construction is possible since $\mathbf{H}$ is \emph{cohesive} as an $\infty$-topos \cite[section 4.4]{schreiber}: this guarantees that the $\infty$-functor $\vert-\vert$ from $\mathbf{H}$ to $\infty$-groupoids does indeed exist, and preserves products.  Notice that the ordinary enrichment of $\mathbf{H}$ is instead given by
%$\mathbf{H}(\Sigma_1,\Sigma_2) = \flat ([\Sigma_1,\Sigma_2])$, where $\flat$ is the right adjoint to $\vert-\vert$. }}.
%%
%\begin{equation}\label{homotopyenrichment}
%\mathbf{H}^\infty(\Sigma_1,\Sigma_2):=\big\vert[\Sigma_1,\Sigma_2]\big\vert
%\end{equation}
%%
%where} $[\,,\,]$ denotes the internal-hom in $\mathbf{H}${, i.e. the mapping $\infty$-stack construction, } and { $\vert X\vert$ is %the topological realization of $X$; under the identification of (nice) topological spaces with $\infty$-groupoids, this is the smooth Poincar\'e $\infty$-groupoid of $X$.}
%Similarly we write $\mathbf{Aut}^\infty(\Sigma)$ for the sub-object of invertible objects in  $\mathbf{H}^\infty(\Sigma,\Sigma)$. {Notice that, since geometric realization preserves products, the composition of { internal homs
%\[
%  [\Sigma_1,\Sigma_2]\times [\Sigma_2,\Sigma_3]\to [\Sigma_1,\Sigma_3]
%\]
%}
%naturally induces a composition
%\[
%\mathbf{H}^\infty(\Sigma_1,\Sigma_2)\times \mathbf{H}^\infty(\Sigma_2,\Sigma_3)\to \mathbf{H}^\infty(\Sigma_1,\Sigma_3).
%\]
}
%%
%\begin{equation}
%  \mathbf{Aut}^\infty(\Sigma):=\Pi([\Sigma,\Sigma])
%\end{equation}
%%
%for the homotopy type of the internal
%automorphism $\infty$-groups, more on this below.

Here is another example. For $G$ a Lie group, we will write $\mathbf{B}G$ for the smooth stack of principal $G$-bundles. This means that for $\Sigma$ a smooth manifold, a morphism $f\colon \Sigma\to \mathbf{B}G$ is precisely a $G$-principal bundle over $\Sigma$. So, in particular, $\mathbf{B}GL(n;\mathbb{R})$ is the smooth stack of principal $GL(n;\mathbb{R})$-bundles. Identifying a principal $GL(n;\mathbb{R})$-bundle with its associated rank $n$ real vector bundle,  $\mathbf{B}GL(n;\mathbb{R})$ is equivalently the smooth stack
of rank $n$ real vector bundles and their isomorphisms. In particular, a map $\Sigma\to \mathbf{B}GL(n;\mathbb{R})$ is precisely the datum of a rank $n$ vector bundle on the smooth manifold $\Sigma$. Again, for a given smooth manifold $\Sigma$, the homotopy type of $\mathbf{H}(\Sigma,\mathbf{B}G)$ is too rigid for our aims: the $\infty$-groupoid $\mathbf{H}(\Sigma,\mathbf{B}G)$ is actually a 1-groupoid. This means that we have objects, which are the principal $G$-bundles over $\Sigma$, and 1-morphism between these objects, which are isomorphisms of principal $G$-bundles, and then nothing else: we do not have nontrivial morphisms between the morphisms, and there's no such a thing like ``a morphism can be smoothly deformed into another morphism'', which again is something very natural to consider from a geometric point of view. Making the smooth structure of the group $G$ come into play we get the following description of
the $\infty$-groupoid $\mathbf{H}^\infty(\Sigma,\mathbf{B}G)$: its objects are the principal $G$-bundles over $\Sigma$ and its 1-morphism are the isomorphisms of principal $G$-bundles, much as for $\mathbf{H}(\Sigma,\mathbf{B}G)$, but then we have also 2-morphisms given by isotopies between isomorphisms, 3-morphisms given by isotopies between isotopies, and so on.
%Of course for $G$ a Lie group then $\mathbf{B}G$ is equivalently the traditional classifying space of that group
%(or rather of its principal bundles). Regarding this homotopy type again as a locally constant stack
%on smooth manifolds, then the operation of regarding points as constant paths provides a
%canonical map of smooth higher stacks
%
%Notice that if we write $G_{\mathrm{disc}}$ for the topological group obtained from the Lie group $G$ by endowing it with the discrete %topology, then $\mathbf{B}G\cong \mathbf{B}G_{\mathrm{disc}}$ and the continuous homomorphism $\mathrm{id}_G\colon G_{\mathrm{disc}}\to %G$ induces a natural
%embedding map
Notice that we have a canonical $\infty$-functor\footnote{{In terms of cohesion this is a component of the canonical
points-to-pieces-transform $\Gamma [\Sigma,\mathbf{B}G] \to [\Sigma,\mathbf{B}G] \to \vert[\Sigma,\mathbf{B}G]\vert$.}}
\begin{equation}
\mathbf{H}(\Sigma,\mathbf{B}G) \longrightarrow \mathbf{H}^\infty(\Sigma,\mathbf{B}G).
\end{equation}
This is nothing but saying that for $j\geq 2$, the $j$-morphisms in $\mathbf{H}(\Sigma,\mathbf{B}G)$ are indeed very special $j$-morphisms in $\mathbf{H}^\infty(\Sigma,\mathbf{B}G)$, namely the identities. Moreover, when $G$ happens to be a discrete group, this embedding is actually an equivalence of $\infty$-groupoids.

%This homotopy-enrichment passes to slices as follows. For $B \in \mathbf{H}$ any object, then the
%slice $\mathbf{H}_{/B}$ is first of all naturally enriched over $\mathbf{H}$ by base changing the
%internal hom in the slice along the right adjoint $\mathbf{H}_{/B}\to \mathbf{H}$ of the functor
%that forms the product with $B$. This $\mathbf{H}$-enrichment may then in turn be homotopified by
%applying $\Pi :\mathbf{H}\to \infty \mathrm{Grpd}$, as before. We will write $\mathbf{H}^\infty_{/X}$
%for the $\infty$-category thus obtained.

\subsection{Geometrically discrete $\infty$-stacks and the homotopy type $\mathscr{B}GL(n)$}
The following notion will be of great relevance for the results of this note. Recall from above the full inclusion
\begin{equation}
\mathrm{LConst}:\infty\mathrm{Grpd} \to \mathbf{H}
\end{equation}
given by regarding an $\infty$-groupoid $\mathcal{G}$ as a constant presheaf over Cartesian spaces. We will say that an object in $\mathbf{H}$ is a \emph{geometrically discrete} $\infty$-stack if it belongs to the essential image of $\mathrm{LConst}$. An example of a geometrically discrete object in $\mathbf{H}$ is given by the 1-stack $\mathbf{B}G$, with $G$ a discrete group. More generally, for $A$ an abelian discrete group the (higher) stacks $\mathbf{B}^n A$ of principal $A$-$n$-bundles are geometrically discrete. The importance of considering geometrically discrete $\infty$-stacks is that the {geometric realization functor $\vert-\vert$ introduced before is left adjoint to $\mathrm{LConst}$}. In particular{, denoting by $\Pi\colon \mathbf{H}\to \mathbf{H}$ the composition $\mathrm{LConst}\circ \vert-\vert$,} we have a canonical unit morphism
\begin{equation}
\mathrm{id}_{\mathbf{H}}\to \Pi
\end{equation}
which is the canonical morphism from a smooth stack to its homotopy type (and which corresponds to looking at points of a smooth manifold $\Sigma$ as constant paths into $\Sigma$). In particular, for $G$ a group, we will write $\mathscr{B}G$ for the homotopy type of $\mathbf{B}G$, i.e., we set $\mathscr{B}G:=\Pi\mathbf{B}G$.
%(Notice that since $\mathrm{LConst}$ is a
% fully faithful inclusion, there is no harm in suppressing it notationally, which we will freely do.)
  This {precisely encodes}
  %is equivalently
  the traditional classifying space $BG$ for the group $G$
(or rather of its principal bundles) {within $\mathbf{H}^\infty$}. { Namely, for $\Sigma$ a smooth manifold we have, by the very definition of adjunction
\[
\mathbf{H}(\Sigma, \mathscr{B}G) =\infty\mathrm{Grpd}(|\Sigma|,|\mathbf{B}G|).
\]
A model for the classifying space $BG$ is precisely given by the topological realization of $\mathbf{B}G$, while $|\Sigma|$ is nothing but the topological space underlying the smooth manifold $\Sigma$ (so that by a little abuse of notation, we will simply write $\Sigma$ for $\Sigma$).
Moreover, since by definition  $\mathscr{B}G$ is geometrically discrete we also have  $\mathbf{H}^{\infty}(\Sigma, \mathscr{B}G) \cong \mathbf{H}(\Sigma, \mathscr{B}G)$, so that in the end we have a natural equivalence
\[
\mathbf{H}^{\infty}(\Sigma, \mathscr{B}G) = \infty\mathrm{Grpd}(\Sigma,BG).
\]
Under the equivalence between (nice) topological spaces and $\infty$-groupoids, on the right we have the $\infty$-groupoid of \emph{continuous} maps from $\Sigma$ to the classifying space $BG$. Notice how this example precisely shows how $\mathbf{H}^\infty$ is a setting where we can talk on the same footing of smooth and continuous phenomena. For instance, smooth maps from a smooth manifold $\Sigma$ to another smooth manifold $M$ and their smooth homotopies are encoded into $\mathbf{H}^{\infty}(\Sigma, M)$, while continuous maps between $\Sigma$ and $M$ and their continuous homotopies  are encoded into $\mathbf{H}^{\infty}(\Sigma, \Pi(M))$.}

The unit $\mathrm{id}_{\mathbf{H}}\to \Pi$ gives a canonical morphism
\begin{equation}
\mathbf{B}G\to \mathscr{B}G,
\end{equation}
which is an equivalence for a discrete group $G$.
This tells us in particular  that any object over $\mathbf{B}G$ is naturally also an object over $\mathscr{B}G$.  For instance (and this example will be the most relevant for what follows), a choice of a rank $n$ vector bundle over a smooth manifold $\Sigma$ realises $\Sigma$ as an object over $\mathscr{B}GL(n;\mathbb{R})$.\\
{Notice how we have a canonical morphism
\begin{equation}
\mathbf{H}(\Sigma,\mathbf{B}G) \longrightarrow  \mathbf{H}^\infty(\Sigma,\mathscr{B}G)
\end{equation}
}
obtained by  composing the canonical morphism $\mathbf{H}(\Sigma,\mathbf{B}G) \to \mathbf{H}^\infty(\Sigma,\mathbf{B}G)$ mentioned in the previous section with the push forward morphism $\mathbf{H}^\infty(\Sigma,\mathbf{B}G) \to \mathbf{H}^\infty(\Sigma,\mathscr{B}G)$,
The main reason to focus on geometrically discrete stacks is that, { though $\vert-\vert$ preserves finite products, it does \emph{not} in general preserve homotopy pullbacks. Neverthless, $\vert-\vert$ does indeed preserve homotopy pullbacks of diagrams
whose tip is a geometrically discrete object in $\mathbf{H}$ \cite[thm. 3.8.19]{schreiber}.}

\subsubsection{Working in the slice}
Let now $n$ be a fixed nonnegative integer and let $0\leq k\leq n$. Any $k$-dimensional smooth manifold $M_k$ comes canonically equipped with a rank $n$ real vector bundle given by the stabilized tangent bundle $T^{\mathrm{st}}M_k=TM_k\oplus \underline{\mathbb{R}}^{n-k}_{M_k}$, where $\underline{\mathbb{R}}^{n-k}_{M_k}$ denotes the trivial rank $(n-k)$ real vector bundle over $M_{k}$. We can think of the stabilised tangent bundle\footnote{To be precise, $T^{\mathrm{st}}$ is the map of stacks induced by the frame bundle of the stabilised tangent bundle to $M_{k}$.} as a morphism
\begin{equation}
M_{k}\xrightarrow{T^{\mathrm{st}}} \mathscr{B}GL(n) %\longrightarrow \mathbf{B}GL(n)
\end{equation}
where $GL(n)$, as in the following, denotes $GL(n;\mathbb{R})$.\\
Namely, we can regard any smooth manifold of dimension at most $n$ as an object \emph{over} $\mathscr{B}GL(n)$. This suggests that a natural setting to work in is the slice topos $\mathbf{H}^\infty_{/\mathscr{B}GL(n)}$, which in the following we will refer to simply as ``the slice'':
%%
%\begin{comment}
%An appropriate definition of $\mathbf{H}^\infty_{/\mathbf{B}GL(n)}$ will appear in \cite{FSV2}
%\end{comment}
%%
 in other words, all objects involved will be equipped with morphisms to $\mathscr{B}GL(n)$, and a morphism between $X\xrightarrow{\varphi} \mathscr{B}GL(n)$ and $Y\xrightarrow{\psi} \mathscr{B}GL(n)$ will be a homotopy commutative diagram
\begin{equation}
\xymatrix{
X\ar[dr]_{\varphi}\ar[rr]^{f}&&Y\ar[dl]^{\psi}\\
&\mathscr{B}GL(n)
 \ar@{=>}(20.25,-3.25);(16.75,-6.75)^{\eta}
}.
\end{equation}
More explicitly, if we denote by $E_\varphi$ and $E_\psi$ the rank $n$ real vector bundles over $X$ and $Y$ corresponding to the morphisms $\varphi$ and $\psi$, respectively, then we see that a morphism in the slice between $X\xrightarrow{\varphi} \mathscr{B}GL(n)$ and $Y\xrightarrow{\psi} \mathscr{B}GL(n)$ is precisely the datum of a morphism $f\colon X\to Y$ together with an \emph{isomorphism} of vector bundles over $X$,
\begin{equation}
\eta:f^*E_\psi \xrightarrow{\simeq} E_\varphi.
\end{equation}
Notice that these are precisely the same objects and morphisms as if we were working in the slice over $\mathbf{B}GL(n)$ in $\mathbf{H}$. Neverthless, as we will see in the following sections, where the use of $\mathbf{H}^\infty$ makes a difference is precisely in allowing nontrivial higher morphisms. Also, the use of the homotopy type $\mathscr{B}GL(n)$ in place of the smooth stack $\mathbf{B}GL(n)$ will allow us to make all constructions work ``up to homotopy'', and to identify, for instance, $\mathscr{B}GL(n)$ with $\mathscr{B}O(n)$.
%Before moving to this, let us see a couple of examples of morphism in the slice over $\mathbf{B}GL(n'\mathbb{R})$.
%%
\begin{example}
The inclusion of the trivial group into $GL(n)$ induces a natural morphism $*\to \mathscr{B}GL(n)$, corresponding to the choice of the trivial bundle. If $M_k$ is a $k$-dimensional manifold, then a morphism
\begin{equation}
\xymatrix{
M_k\ar[dr]_{T^{\mathrm{st}}}\ar[rr]&&{*}\ar[dl]\\
&\mathscr{B}GL(n)
 \ar@{=>}(20.25,-3.25);(16.75,-6.75)^{\eta}
}
\end{equation}
is precisely a trivialisation of the stabilised tangent bundle of $M_k$, i.e., an $n$-framing of $M$.
\end{example}
\begin{example}
Let $X$ be a smooth manifold, and let $\zeta$ be a rank $n$ real vector bundle over $X$, which we can think of as a morphism $\rho^{}_\zeta\colon X\to \mathscr{B}GL(n)$. Then a morphism
\begin{equation}
\xymatrix{
M_k\ar[dr]_{T^{\mathrm{st}}}\ar[rr]^{f}&&X\ar[dl]^{\rho^{}_\zeta}\\
&\mathscr{B}GL(n)
 \ar@{=>}(20.25,-3.25);(16.75,-6.75)^{\eta}
}
\end{equation}
is precisely the datum of a smooth map $f\colon M_k\to X$ and of an isomorphism $\eta\colon f^*\zeta\to TM\oplus \underline{\mathbb{R}}^{n-k}_{M_k}$. These are the data endowing $M_k$ with a $(X,\zeta)$-structure in the terminology of \cite{lurieTFT}.
\end{example}
The examples above suggest to allow $X$ to be not only a smooth manifold, but a smooth $\infty$-stack.
%\footnote{This is what Lurie seems to suggest in the definition of $(X,\zeta)$-structures in \cite{lurieTFT}, where he allows $X$ to be %any topological space, and requires the morphism $f$ only to be continous.}
While choosing such a general target $(X,\zeta)$ could at first seem like a major abstraction, this is actually what one commonly encounters in everyday mathematics.
For instance a lift through $\mathbf{B}O(n)\to \mathbf{B}GL(n)$ is precisely a ($n$-stable) Riemannian structure.
%,while a lift through $\mathscr{B}O(n) \to \mathscr{B}GL(n)$ is no structure at all.
Generally, for $G \hookrightarrow GL(n)$ any inclusion of Lie groups, or even more generally for $G \to GL(n)$ any morphism of Lie groups, then a lift through
$\mathbf{B}G \to  \mathbf{B}GL(n)$ is a ($n$-stable) \emph{$G$-structure}, e.g., an almost symplectic structure,
an almost complex structure, etc. (one may also phrase integrable $G$-structures in terms of slicing, using more
of the axioms of cohesion than we need here).
For instance, the inclusion of the connected component of the identity $GL^+(n)\hookrightarrow GL(n)$ corresponds to a morphism of higher stacks $\iota\colon \mathbf{B}GL^+(n)\to \mathbf{B}GL(n)$, and a morphism in the slice from $(M_k,T^{\mathrm{st}})$ to $(\mathbf{B}GL^+(n),\iota)$ is precisely the choice of a (stabilised) orientation on $M_k$.
%Similarly, the inclusion of $O(n)$ into $GL(n)$ induces the morphism of smooth higher stacks (actually an equivalence) $\iota\colon
%\mathscr{B}O(n)\to \mathscr{B}GL(n)$, so that morphisms in the slice from $(M_k,T^{\mathrm{st}})$ to $(\mathscr{B}O(n),\iota)$
%correspond to equipping $M_k$ with a (stabilised) Euclidean structure. Along the same lines, one can formulate  (stabilised)
%Minkowskian structures,
For $G$ a higher connected cover of $O(n)$ then lifts through
 $\mathbf{B}G \to \mathbf{B}O(n)\to \mathbf{B}GL(n)$
 correspond to spin structures, string structures, etc.\\
  On the other hand, since $\mathscr{B}O(n) \to \mathscr{B}GL(n)$ is an equivalence, a lift through  $\mathscr{B}O(n) \to \mathscr{B}GL(n)$ is no additional structure on a smooth manifold $M_k$, and the stabilized tangent bundle of $M_k$ can be equally seen as a morphism to $\mathscr{B}O(n)$. Similarly, for $G \to GL(n)$ any morphism of Lie groups, lifts of $T^\mathrm{st}$ through $\mathscr{B}G \to \mathscr{B}GL(n)$ correspond to ($n$-stable) \emph{topological} $G$-structures.

\subsection{From homotopy group actions to objects in the slice}
We will mainly be interested in objects of $\mathbf{H}^\infty_{/\mathscr{B}GL(n)}$ obtained as a homotopy group action of a smooth (higher) group $G$ on some stack $X$, when $G$ is equipped with a $\infty$-group morphism to $GL(n)$. We consider then the following
\begin{definition}
{A homotopy action of a smooth $\infty$-group $G$ on $X$ is the datum of a smooth $\infty$-stack $X/\!/_hG$ together with
%a morphism $\rho\colon X/\!/_hG\to \mathscr{B}G$ satisfying the following
a homotopy pullback diagram
\begin{equation}\label{Gaction}
\xymatrix{
X\ar[d]\ar[r]&X/\!/_hG\ar[d]_{\rho}\\
{*}\ar[r] &\mathscr{B}G
}
\end{equation}
}
\end{definition}
\noindent
Unwinding the definition, one sees that a homotopy action of $G$ is nothing but an action of the homotopy type of $G$ and that $X/\!/_hG$ is realised as the stack quotient {$X/\!/\Pi G$}. See \cite{NSSa} for details. Since $G$ is equipped with a smooth group morphism to $GL(n)$, and since this induces a morphism of smooth stacks $\mathscr{B}G\to \mathscr{B}GL(n)$, the stack $X/\!/_hG$ is naturally an object over $\mathscr{B}GL(n)$. In particular, when $X$ is a deloopable object, i.e., when there exists a stack $Y$ such that $\Omega Y\cong X$, then one obtains a homotopy $G$-action out of any morphism $c\colon \mathscr{B}G\to Y$. Indeed, in this situation one can define $X/\!/_hG\to \mathscr{B}G$ by the homotopy pullback
\begin{equation}
\xymatrix{
X/\!/_hG\ar[d]_{\rho_c}\ar[r]&{*}\ar[d]\\
\mathscr{B}G\ar[r]^{c} &Y
}
\end{equation}
%A particular example of this construction we will be interested in is the following.
%%
By using the pasting law for homotopy pullbacks, we can see that $X$, $X/\!/_hG$, and the morphism $\rho_{c}$ fit in a homotopy pullback diagram as in (\ref{Gaction}).
\begin{example}\label{charclass}
Let $c$ be a degree $d+1$ characteristic class for the group $SO(n)$. Then $c$ can be seen as the datum of a morphism of stacks $c\colon \mathscr{B}SO(n)\to \mathscr{B}^{d+1}\mathbb{Z}\cong \mathbf{B}^{d+1}\mathbb{Z}$, where $\mathbf{B}^{d+1}\mathbb{Z}$ is the smooth stack associated by the Dold-Kan correspondence to the chain complex with $\mathbb{Z}$ concentrated in degree $d+1$, i.e., the stack (homotopically) representing degree $d+1$ integral cohomology. Notice how the discreteness of the abelian group $\mathbb{Z}$ came into play to give the equivalence $\mathscr{B}^{d+1}\mathbb{Z}\cong \mathbf{B}^{d+1}\mathbb{Z}$. Since we have $\Omega \mathbf{B}^{d+1}\mathbb{Z}\cong \mathbf{B}^d\mathbb{Z}$, the characteristic class $c$ defines a homotopy action
\begin{equation}
\rho_c\colon \mathbf{B}^d\mathbb{Z}/\!/_hSO(n) \to \mathscr{B}SO(n)
\end{equation}
and so %, by composing with  $\mathbf{B}SO(n)\to \mathbf{B}GL(n)$,
 it  realises $\mathbf{B}^d\mathbb{Z}/\!/_hSO(n)$ as an object in the slice $\mathbf{H}^\infty_{/\mathscr{B}GL(n)}$. For instance, the first Pontryagin class $p_1$ induces a homotopy action
\begin{equation}
\rho_{p_1}\colon \mathbf{B}^3\mathbb{Z}/\!/_hSO(n) \to \mathscr{B}SO(n).
\end{equation}
\end{example}
\section{$\rho$-framed manifolds and their automorphisms $\infty$-group}\label{rhoframed}
We can now introduce the main definition in the present work.
\begin{definition}
Let $M$ be a $k$-dimensional manifold, and let $\rho\colon X\to  \mathscr{B}GL(n)$ be a morphisms of smooth $\infty$-stacks, with $k\leq n$. Then a \emph{$\rho$-framing} (or \emph{$\rho$-structure}) on $M$ is a lift of the stabilised tangent bundle seen as a morphism $T^{\mathrm{st}}\colon M\to \mathscr{B}GL(n)$ to a morphism $\sigma\colon M\to X$, namely a homotopy commutative diagram of the form
\begin{equation}
\xymatrix{
M\ar[dr]_{T^{\mathrm{st}}}\ar[rr]^{\sigma}&&X\ar[dl]^{\rho}\\
&\mathscr{B}GL(n)
 \ar@{=>}(20.25,-3.25);(16.75,-6.75)^{\eta}
}
\end{equation}
\end{definition}
By abuse of notation, we will often say that the morphism $\sigma$ is the $\rho$-framing, omitting the explicit reference to the homotopy $\eta$, which is, however, always part of the data of a $\rho$-framing.\\
Since the morphism $\rho\colon X\to \mathscr{B}GL(n)$ is an object in the slice $\mathbf{H}^\infty_{/_{\mathscr{B}GL(n)}}$, we can consider the slice over $\rho$:  $(\mathbf{H}^\infty_{/_{\mathscr{B}GL(n)}})_{/_\rho}$. Although this double slice may seem insanely abstract at first, it is something very natural. Its objects are homotopy commutative diagrams, namely 2-simplices
\begin{equation}
\xymatrix{
Y\ar[dr]_{\tilde{\rho}}\ar[rr]^{a}&&X\ar[dl]^{\rho}\\
&\mathscr{B}GL(n)
 \ar@{=>}(20.25,-3.25);(16.75,-6.75)^{\eta}
}
\end{equation}
while its morphisms are homotopy commutative 3-simplices
\begin{equation}
\begin{xy}
0;/r.34pc/:
,(0,0)*{Y}
,(10,-5)*{Z}
,(30,3)*{X}
,(12,-20)*{\mathscr{B}GL(n)}
,(1.5,-1);(8.5,-4.2)**\dir{-}?>* \dir{>}
,(12,-4);(28.3,1.5)**\dir{-}?>* \dir{>}
,(1.5,0);(28.1,2.5)**\dir{-}?>* \dir{>}
,(29.5,1);(13,-17.5)**\dir{-}?>* \dir{>}
,(1,-2);(11,-17.5)**\dir{-}?>* \dir{>}
,(10,-7);(12,-17.2)**\dir{-}?>* \dir{>}
,(22.5,-9)*{\scriptstyle{\rho}}
,(5,-11)*{\scriptstyle{\tilde{\rho}}}
,(9,-9)*{\scriptstyle{\hat{\rho}}}
,(14.5,2.5)*{\scriptstyle{a}}
,(17,-4)*{\scriptstyle{b}}
,(6.5,-1.2)*{\scriptstyle{f}}
\end{xy}
\end{equation}
where for readability we have omitted the homotopies decorating the faces and the interior of the 3-simplex, and similarly, additional data must be provided for higher morphisms.\\
In particular we see that a $\rho$-framing $\sigma$ on $M$ is naturally an object in the double slice $(\mathbf{H}_{/_{\mathscr{B}GL(n)}})_{/_\rho}$. Moreover, the collection of all $k$-dimensional $\rho$-framed manifolds has a natural $\infty$-groupoid structure which is compatible with the forgetting of the framing, and with the fact that any $\rho$-framed manifold is in particular an object in the double slice $(\mathbf{H}^\infty_{/_{\mathscr{B}GL(n)}})_{/_\rho}$. More precisely, let $\mathscr{M}_k$ denote the $\infty$-groupoid whose objects are $k$-dimensional smooth manifolds, whose 1-morphisms are diffeomorphisms of $k$-dimensional manifolds whose 2-morphisms are isotopies of diffeomorphisms, and so on\footnote{The $\infty$-groupoid $\mathscr{M}_{k}$ can be rigorously defined as  $\Omega(\rm{Cob}_{t}(k))$, where $\rm{Cob}_{t}(k)$ is the $(\infty,1)$-category defined in \cite{lurieTFT} in the context of topological field theory.}. There is then an $\infty$-groupoid $\mathscr{M}_k^\rho$ of $\rho$-framed $k$-dimensional manifolds which is a $\infty$-subcategory of $(\mathbf{H}^\infty_{/_{\mathscr{B}GL(n)}})_{/_\rho}$, and comes equipped with a forgetful $\infty$-functor
\begin{equation}
\mathscr{M}_k^\rho\to \mathscr{M}_k.
\end{equation}
Namely, since the differential of a diffeomorphism between $k$-dimensional manifolds $M$ and $N$ can naturally be seen as an invertible 1-morphism between $M$ and $N$ as objects over $\mathbf{B}GL(n)$, we have a natural (not full) embedding
\begin{equation}
\mathscr{M}_k\hookrightarrow \mathbf{H}^\infty_{/\mathscr{B}GL(n)} .
\end{equation}
Consider then the forgetful functor
\begin{equation}
(\mathbf{H}^\infty_{/_{\mathscr{B}GL(n)}})_{/_\rho}\to \mathbf{H}^\infty_{/\mathscr{B}GL(n)}
\end{equation}
We have then the following important
\begin{definition}
Let $\rho\colon X\to \mathscr{B}GL(n)$ be an object in $\mathbf{H}^\infty_{/_{\mathscr{B}GL(n)}}$.
The $\infty$-groupoid $\mathscr{M}_k^\rho$ is then defined as the homotopy pullback diagram
\begin{equation}
\xymatrix{
\mathscr{M}_k^\rho\ar[r]\ar[d]& (\mathbf{H}^\infty_{/_{\mathscr{B}GL(n)}})_{/_\rho}\ar[d]\\
\mathscr{M}_k\ar[r]&\mathbf{H}^\infty_{/\mathscr{B}GL(n)}
}
\end{equation}
\end{definition}
Given two $\rho$-framed $k$-dimensional manifolds $(M,\sigma,\eta)$ and $(N,\tau,\vartheta)$, the $\infty$-groupoid $\mathscr{M}_k^\rho((M,\sigma,\eta),(N,\tau,\vartheta))$ is the homotopy pullback
\begin{equation}
\xymatrix{
\mathscr{M}_k^\rho((M,\sigma,\eta),(N,\tau,\vartheta))\ar[r]\ar[d]& (\mathbf{H}^\infty_{/_{\mathscr{B}GL(n)}})_{/_\rho}(\sigma,\tau)\ar[d]\\
\mathscr{M}_k(M,N)\ar[r]&\mathbf{H}^\infty_{/\mathscr{B}GL(n)}(T^{\mathrm{st}}_M,T^{\mathrm{st}}_N)
}
\label{DefiningHomotopyPullback}
\end{equation}
In particular, if we denote with $\mathrm{Diff}(M)$ the $\infty$-groupoid of diffeomorphisms of $M$, namely the automorphism $\infty$-group of $M$ as an object in $\mathscr{M}_k$, and we accordingly write $\mathrm{Diff}^\rho(M,\sigma)$ for the automorphisms $\infty$-group of $(M,\sigma)$ as an object in $\mathscr{M}^\rho_k$, then we have a homotopy pullback
\begin{equation}\label{def.diff}
    \xymatrix{
      \mathrm{Diff}^\rho(M,\sigma,\eta)
      \ar[r]
      \ar[d]
      &
      \mathbf{Aut}^\infty_{/\rho}(\sigma)
      \ar[d]
      \\
      \mathrm{Diff}(M)
      \ar[r]
      &
      \mathbf{Aut}^\infty_{/\mathscr{B}GL(n)}(T^{\mathrm{st}}_M)
    }
 \end{equation}
 where $\mathbf{Aut}_{(-)}^\infty(-)$ denotes the homotopy type of the relevant $\mathbf{H}$-internal automorphisms $\infty$-group. In particular, to abbreviate the notation, we will denote with $\mathbf{Aut}^\infty_{\rho}(\sigma)$ the automorphism $\infty$-group of $\sigma$ in $(\mathbf{H}^\infty_{/_{\mathscr{B}GL(n)}})_{/_\rho}$.\\
More explicitly, an element in $\mathrm{Diff}^\rho(M,\sigma,\eta)$ is a diffeomorphism $\varphi\colon M\to M$ together with an isomorphism $\alpha\colon \varphi^*\sigma\xrightarrow{\simeq}\sigma$, and a filler $\beta$ for the 3-simplex
\begin{equation}
\begin{xy}
0;/r.34pc/:
,(-5,0)*{M}
,(10,-5)*{M}
,(30,3)*{X}
,(12,-20)*{\mathscr{B}GL(n)}
,(-3,-1);(8.5,-4.2)**\dir{-}?>* \dir{>}
,(12,-4);(28.3,1.5)**\dir{-}?>* \dir{>}
,(-3,0);(28.1,2.5)**\dir{-}?>* \dir{>}
,(29.5,1);(13,-17.5)**\dir{-}?>* \dir{>}
,(-3,-2);(11,-17.5)**\dir{-}?>* \dir{>}
,(10,-7);(12,-17.2)**\dir{-}?>* \dir{>}
,(22.5,-9)*{\scriptstyle{\rho}}
,(1,-11)*{\scriptstyle{T^{st}}}
,(13,-9)*{\scriptstyle{T^{st}}}
,(5,-6)*{\scriptstyle{d\varphi}}
,(5,-8)*{\rotatebox[origin=c]{240}{$\Rightarrow$}}
,(14.5,2.5)*{\scriptstyle{\sigma}},
(13.0,-0.0)*{\scriptstyle{\alpha}},
(14.0,-1.5)*{\rotatebox[origin=c]{50}{$\Rightarrow$}},
,(17,-4)*{\scriptstyle{\sigma}}
,(18,-6)*{\scriptstyle{\eta}}
,(18,-8)*{\rotatebox[origin=c]{240}{$\Rightarrow$}}
,(6.5,-1.2)*{\scriptstyle{\varphi}}
\end{xy}
\end{equation}
\subsection{Functoriality and homotopy invariance of $\mathscr{M}_k^\rho$} In this section we will explore some of the properties of $\mathscr{M}_k^\rho$, which will be useful in the following.\\
It immediately follows from the definition that the forgetful functor $\mathscr{M}_k^\rho\to \mathscr{M}_k$ is a equivalence for $\rho\colon X\to \mathscr{B}GL(n)$ an equivalence in $\mathbf{H}^{\infty}(X, \mathscr{B}GL(n))$. In particular, if $\rho$ is the identity morphism of $\mathscr{B}GL(n)$ and we write $\mathscr{M}_k^{GL(n)}$ for $\mathscr{M}_k^{\mathrm{id}_{\mathscr{B}GL(n)}}$ then we have $\mathscr{M}_k^{GL(n)}\cong \mathscr{M}_k$. Less trivially, if $X=\mathscr{B}O(n)$, and $\rho$ is the natural morphism
\begin{equation}
\iota_{O(n)}\colon \mathscr{B}O(n)\to \mathscr{B}GL(n)
\end{equation}
 induced by the inclusion of $O(n)$ in $GL(n)$, then $\rho$ is again an equivalence, and we get $\mathscr{M}_k^{O(n)}\cong \mathscr{M}_k$, where we have denoted $\mathscr{M}_k^{\iota_{O(n)}}$ with $\mathscr{M}_k^{O(n)}$.\\
More generally, if $\rho$ and $\tilde{\rho}$ are equivalent objects in the slice $\mathbf{H}^{\infty}_{/\mathscr{B}GL(n)}$, then we have equivalent $\infty$-groupoids $\mathscr{M}_k^\rho$ and $\mathscr{M}_k^{\tilde{\rho}}$. For instance, the inclusion of $SO(n)$ into $GL(n)^+$ induces an equivalence between $\mathscr{B}SO(n)$ and $\mathscr{B}GL(n)^+$ over $\mathscr{B}GL(n)$, and so we have a natural equivalence $\mathscr{M}_k^{SO(n)}\cong \mathscr{M}_k^{GL(n)^+}$. Since the objects in the $\infty$-groupoid $\mathscr{M}_k^{GL(n)^+}$ are $k$-dimensional manifolds whose stabilised tangent bundle is equipped with a lift to an $SO(n)$-bundle, the objects of $\mathscr{M}_k^{GL(n)^+}$ are oriented $k$-manifolds. Moreover the pullback defining $\mathscr{M}_k^{GL(n)^+}$ precisely picks up oriented diffeomorphisms, hence the forgetful morphism $\mathscr{M}_k^{GL(n)^+}\to\mathscr{M}_k$ induces an equivalence between $\mathscr{M}_k^{GL(n)^+}$ and the $\infty$-groupoid $\mathscr{M}_k^{\mathrm{or}}$ of oriented $k$-dimensional manifolds with orientation preserving diffeomorphisms between them. As a consequence, one has a natural equivalence
%%%
\begin{equation}
\mathscr{M}_k^{SO(n)}\cong\mathscr{M}_k^{\mathrm{or}}
\end{equation}
%%%
Let $\psi\colon \rho\to \tilde{\rho}$ be a morphism in the slice $\mathbf{H}^{\infty}_{/\mathscr{B}GL(n)}$ between $\rho\colon X\to\mathscr{B}GL(n)$ and $\tilde{\rho}\colon Y\to \mathscr{B}GL(n)$. Then one has an induced push-forward morphism
\begin{equation}
\psi_*\colon \mathscr{M}_k^\rho\to \mathscr{M}_k^{\tilde{\rho}},
\end{equation}
which (by (\ref{DefiningHomotopyPullback}), and using the pasting law) fits into the homotopy pullback diagram
\begin{equation}
\xymatrix{
\mathscr{M}_k^\rho\ar[r]\ar[d]_{\psi_*}& (\mathbf{H}^\infty_{/_{\mathscr{B}GL(n)}})_{/_\rho}\ar[d]^{\Psi_*}\\
\mathscr{M}_k^{\tilde{\rho}}\ar[r]&(\mathbf{H}^\infty_{/_{\mathscr{B}GL(n)}})_{/_ {\tilde{\rho}}
}}
\end{equation}
where $\Psi_*$ denotes the base changing $\infty$-functor on the slice topos.\\
The homotopy equivalences illustrated above are particular cases of this functoriality: indeed, when $\psi$ is invertible, then $\psi_*$ is invertible as well (up to coherent homotopies, clearly).\\
Recall from Example \ref{charclass} that for any characteristic class $c$ of $SO(n)$ we obtain an object $\rho_{c}$ in the slice $\mathbf{H}^\infty_{/\mathscr{B}GL(n)}$.
In this way we obtain natural morphisms $\mathscr{M}_k^{\rho_{c}}\to \mathscr{M}_k^{SO(n)}$ . In particular, by considering the first Pontryagin class $p_1\colon \mathscr{B}SO(n)\to \mathbf{B}^4\mathbb{Z}$, we obtain a canonical morphism
\begin{equation}
 \mathscr{M}_k^{\rho_{p_1}}\to \mathscr{M}_k^{\mathrm{or}}.
\end{equation}
\subsection{Extensions of $\rho$-diffeomorphism groups}\label{section.extensions}
We are now ready for the extension theorem, which is the main result of this note. Not to break the flow of the exposition, we will postpone the details of the proof to the Appendix.
\\
Let
\begin{equation}
\xymatrix{
X\ar[rd]_{\rho}\ar[rr]^\psi && Y\ar[dl]^{\tilde{\rho}}\\
&\mathscr{B}GL(n)
 \ar@{=>}(20.25,-3.25);(16.75,-6.75)^{\Psi}}
\end{equation}
be a morphism in the slice over $\mathscr{B}GL(n)$, as at the end of the previous section, and let
\begin{equation}
\xymatrix{
M\ar[rd]_{T^{\mathrm{st}}_M}\ar[rr]^\tau && Y\ar[dl]^{\tilde{\rho}}\\
&\mathscr{B}GL(n)
 \ar@{=>}(20.25,-3.25);(16.75,-6.75)^{T}
 }
\end{equation}
be a $\tilde{\rho}$-structure on $M$.
Then, arguing as in Section \ref{rhoframed}, associated to any lift
\begin{equation}
\begin{xy}
0;/r.34pc/:
,(-5,0)*{M}
,(10,-5)*{Y}
,(30,3)*{X}
,(12,-20)*{\mathscr{B}GL(n)}
,(-3,-1);(8.5,-4.2)**\dir{-}?>* \dir{>}
,(28.3,1.5);(12,-4)**\dir{-}?>* \dir{>}
,(-3,0);(28.1,2.5)**\dir{-}?>* \dir{>}
,(29.5,1);(13,-17.5)**\dir{-}?>* \dir{>}
,(-3,-2);(11,-17.5)**\dir{-}?>* \dir{>}
,(10,-7);(12,-17.2)**\dir{-}?>* \dir{>}
,(22.5,-9)*{\scriptstyle{\rho}}
,(1,-11)*{\scriptstyle{T^{st}}}
,(13,-9)*{\scriptstyle{\tilde{\rho}}}
,(5,-6)*{\scriptstyle{T}}
,(5,-8)*{\rotatebox[origin=c]{240}{$\Rightarrow$}}
,(14.5,2.5)*{\scriptstyle{\sigma}},
(13.0,-0.0)*{\scriptstyle{\alpha}},
(14.0,-1.5)*{\rotatebox[origin=c]{195}{$\Rightarrow$}},
,(17,-4)*{\scriptstyle{\psi}}
,(18,-6)*{\scriptstyle{\Psi}}
,(18,-8)*{\rotatebox[origin=c]{-5}{$\Rightarrow$}}
,(5.5,-2.1)*{\scriptstyle{\tau}}
\end{xy}
\end{equation}
(where we are not displaying the label $\Sigma$ on the back face, nor the filler $\beta$ of the 3-simplex)
of $T$ to a $\rho$-structure $\Sigma$ on $M$, we have a homotopy pullback diagram
\begin{equation}
    \xymatrix{
      \mathrm{Diff}^\rho(M,\Sigma)
      \ar[r]
      \ar[d]_{\psi_*}
      &
      \mathbf{Aut}^{\infty}_{/\rho}(\Sigma)
      \ar[d]^{\psi_*}
      \\
      \mathrm{Diff}^{\tilde\rho}(M,T)
      \ar[r]
      &
      \mathbf{Aut}^{\infty}_{/\tilde{\rho}}(T)
    }
 \end{equation}
By the pasting law for homotopy pullbacks and from the pasting of homotopy pullback diagrams we have the following homotopy diagram (see Appendix for the proof)
\begin{equation}\label{diag1}
    \xymatrix{
     \Omega_{\beta}(\mathbf{H}^\infty_{/\mathscr{B}GL(n)})_{/\tilde{\rho}}(T,\Psi)\ar[r] \ar[d]&
      \Omega_{\Sigma}\mathbf{H}^\infty_{/\mathscr{B}GL(n)}(T^{\mathrm{st}}_M,\rho)
      %\Omega_\sigma((\mathbf{H}^{\infty}_{/Y}
      %_{/\mathscr{B}GL(n)})_{/\tilde{\rho}}
      %((M,\tau),(X,\psi)))
      \ar[r]
      \ar[d]
      &
      \mathbf{Aut}^{\infty}_{/\rho}(\Sigma)
      \ar[d]^{\psi_*}
      \\
   {*}\ar[r] &  \Omega_{T}\mathbf{H}^\infty_{/\mathscr{B}GL(n)}(T^{\mathrm{st}}_M,\tilde{\rho})
      \ar[d]
      \ar[r]
      &
      \mathbf{Aut}^{\infty}_{/\tilde{\rho}}(T)\ar[d]\\
&{*}\ar[r]& \mathbf{Aut}^\infty_{/\mathscr{B}GL(n)}(T^{\mathrm{st}}_M)    }
 \end{equation}
We therefore obtain the homotopy pullback diagram
\begin{equation}\label{diag2}
    \xymatrix{
     \Omega_{\beta}(\mathbf{H}^\infty_{/\mathscr{B}GL(n)})_{/\tilde{\rho}}(T,\Psi)
    %  \Omega_\sigma((\mathbf{H}^{\infty}_{/Y})_{/\psi})%_{/\mathscr{B}GL(n)})_{/\tilde{\rho}}
      %((M,\tau),(X,\psi)))
      \ar[r]
      \ar[d]
      &\mathrm{Diff}^\rho(M,\Sigma)
      \ar[d]^{\psi_*}
            \\
      {*}
      \ar[r]
      & \mathrm{Diff}^{\tilde\rho}(M,T)
    }
 \end{equation}
presenting $\mathrm{Diff}^\rho(M,\Sigma)$ as an extension of $\mathrm{Diff}^{\tilde\rho}(M,T)$ by the $\infty$-group
$ \Omega_{\beta}(\mathbf{H}^\infty_{/\mathscr{B}GL(n)})_{/\tilde{\rho}}(T,\Psi)$, i.e., by the loop space (at a given lift $\beta$) of the space $(\mathbf{H}^\infty_{/\mathscr{B}GL(n)})_{/\tilde{\rho}}(T,\Psi)$ of lifts of the $\tilde{\rho}$-structure $T$ on $M$ to a $\rho$-structure $\Sigma$.
%$ \Omega_\sigma((\mathbf{H}^{\infty}_{/Y})_{/\psi})$.
%$\Omega_\sigma((\mathbf{H}^{\infty}_{/Y}
%_{/\mathscr{B}GL(n)})_{/\tilde{\rho}}
%%
Now notice that, by the Kan condition, we have a natural homotopy equivalence
\begin{equation}\label{eq.kan}
(\mathbf{H}^\infty_{/\mathscr{B}GL(n)})_{/\tilde{\rho}}(T,\Psi)\cong \mathbf{H}^\infty_{/Y}(\tau,\psi).
\end{equation}
Namely, since $T$ and $\Psi$ are fixed, the datum of the filler $\alpha$ is homotopically equvalent to the datum of the full 3-simplex, as $T,\Psi$ and $\alpha$ together give the datum of the horn at the vertex $Y$. As a consequence we see that the space of lifts of the $\tilde{\rho}$-structure $T$ to a $\rho$-structure $\Sigma$ is homotopy equivalent to the space
of lifts
\begin{equation}
\xymatrix{
&X\ar[d]^{\psi}\\
M\ar[r]^{\tau}\ar[ru]^{\sigma}&Y
\ar@{=>}(11.25,-6.25);(7.75,-9.75)^{\alpha}
}
\end{equation}
of $\tau$ to a morphism $\sigma\colon M\to X$. We refer the the reader to the Appendix for a rigorous proof of equivalence (\ref{eq.kan}).\\
The arguments above lead directly to
\begin{proposition}\label{proposition.extension}
Let $\rho:X\to{\mathscr{B}GL(n)}$ and $\tilde{\rho}:Y\to{\mathscr{B}GL(n)}$ be morphisms of $\infty$-stacks, and let $(\psi,\Psi):\rho\to{\tilde{\rho}}$ be a morphism in $\mathbf{H}^{\infty}_{/\mathscr{B}GL(n)}$. Let $(M,T)$ be a $\tilde{\rho}$-framed manifold, and let $\Sigma$ be a $\rho$-structure on $M$ lifting $T$ through $(\alpha,\beta)$. We have then the following homotopy pullback
\begin{equation}\label{diag3}
    \xymatrix{
     \Omega_{\alpha}\mathbf{H}^\infty_{/Y}(\tau,\psi)
    %  \Omega_\sigma((\mathbf{H}^{\infty}_{/Y})_{/\psi})%_{/\mathscr{B}GL(n)})_{/\tilde{\rho}}
      %((M,\tau),(X,\psi)))
      \ar[r]
      \ar[d]
      &\mathrm{Diff}^\rho(M,\Sigma)
      \ar[d]^{\psi_*}
            \\
      {*}
      \ar[r]
      & \mathrm{Diff}^{\tilde\rho}(M,T)
    }
 \end{equation}
 \end{proposition}
 \begin{proof}
 Combine diagram (\ref{diag2}) with equivalence (\ref{eq.kan}), which preserves homotopy pullbacks.
 \end{proof}
\begin{remark} Proposition \ref{proposition.extension} gives a presentation of $\mathrm{Diff}^\rho(M,\Sigma)$ as an extension of $\mathrm{Diff}^{\tilde\rho}(M,T)$ by the $\infty$-group
$ \Omega_{\alpha}\mathbf{H}^\infty_{/Y}(\tau,\psi)$. Notice how, for $(T,\tau)$ the identity morphism, i.e.
\begin{equation}
\xymatrix{
Y\ar[rd]_{\tilde{\rho}}\ar[rr]^{\mathrm{id}_Y} && Y\ar[dl]^{\tilde{\rho}}\\
&\mathscr{B}GL(n)
 \ar@{=>}(20.25,-3.25);(16.75,-6.75)^{\mathrm{Id}}
 }
\end{equation}
the space $\mathbf{H}^\infty_{/Y}(\tau,\mathrm{id}_Y)$ is contractible since $\mathrm{id}_Y$ is the terminal object in the slice $\mathbf{H}^\infty_{/Y}$ and so one finds that the extension of $\mathrm{Diff}^{\tilde\rho}(M,T)$ is the trivial one in this case, as expected.
\end{remark}
%\begin{remark}
%In diagrams (\ref{diag1}) and (\ref{diag2}), we %regard $\sigma$ as a $\tilde{\rho}$-structure on $M$. Moreover, we
%have omitted the  higher morphisms realizing the commutativity of the various diagrams, which though play an essential role.
%\end{remark}
%%
%%\subsection{The case of homotopy fibres}
\section{Lifting $\rho$-structures along homotopy fibres}
\label{section.lifting}
\label{homotopy-fibres}
In this section we will investigate a particularly simple and interesting case of the lifting procedure of $\rho$-structures, and
of extensions of $\rho$-diffeomorphisms $\infty$-groups, namely the case when $\psi\colon X\to Y$ is the homotopy fibre in $\mathbf{H}^\infty$ of a morphism $c\colon Y\to Z$ from $Y$ to some pointed stack $Z$.
\\
In this case, by the universal property of the homotopy pullback, the space  $\mathbf{H}^\infty_{/Y}(\tau,\psi)$ of lifts of the $\tilde{\rho}$-structure $\tau$ to a $\rho$-structure $\sigma$ is given by the space of homotopies between the composite morphism $c\circ \tau$ and the trivial morphism $M\to Z$ given by the constant map on the marked point of $Z$:
\begin{equation}
\xymatrix{
M\ar@/^1pc/[rrd]\ar@/_1pc/[rdd]_{\tau}\ar@{-->}[dr]^{\sigma}\\
&X\ar[r]\ar[d]_{\psi}& {*}\ar[d]\\
&Y\ar[r]^c& Z
}
\end{equation}
This fact has two important consequences:
\begin{itemize}
\item a lift $\sigma$ of $\tau$ exists if and only if the class of $c\circ \tau$ in $\pi_0\mathbf{H}^{\infty}(M,Z)$ is the trivial class (the class of the constant map on the marked point $z$ of $Z$);
\item  when a lift exists, the space $\mathbf{H}^{\infty}_{/Y}(\tau,\psi)$ is a torsor for the $\infty$-group of self-homotopies of the constant map $M\to Z$, i.e., for the $\infty$-group object $\Omega\mathbf{H}^{\infty}(M,Z)$. In particular, as soon as $\mathbf{H}^{\infty}_{/Y}(\tau,\psi)$ is nonempty, any lift $\sigma$ of $\tau$ induces an equivalence of $\infty$-groupoids $\mathbf{H}^{\infty}_{/Y}(\tau,\psi)\cong \Omega\mathbf{H}^\infty(M, Z)$ and so an equivalence
\begin{equation}
\Omega_\alpha\mathbf{H}^\infty_{/Y}(\tau,\psi)\cong \Omega^2\mathbf{H}^\infty(M,  Z).
\end{equation}
\end{itemize}
\vskip .3 cm
\noindent
Moreover, as soon as $(Z,z)$ is a geometrically discrete pointed $\infty$-stack, we have $\Omega \mathbf{H}^\infty(M,Z)\cong \mathbf{H}^\infty(M,\Omega Z)$, where $\Omega Z$ denotes the loop space of $Z$ in $\mathbf{H}$ at the distinguished point $z$. In other words, for a geometrically discrete $\infty$-stack $Z$, the loop space of $Z$ in $\mathbf{H}$ also provides a loop space object for $Z$ in $\mathbf{H}^\infty$. Namely, by definition of $\mathbf{H}^\infty$, showing that
\begin{equation}
\xymatrix{
\mathbf{H}^\infty(W,\Omega Z)\ar[r]\ar[d]& {*}\ar[d]\\
{*}\ar[r] & \mathbf{H}^\infty(W, Z)
}
\end{equation}
is a homotopy pullback of $\infty$-groupoids for any $\infty$-stack $W$ amounts to showing that
{
\begin{equation}
\xymatrix{
\big\vert[W,\Omega Z]\big\vert\ar[r]\ar[d]& {*}\ar[d]\\
{*}\ar[r] & \big\vert[W, Z]\big\vert
}
\end{equation}
}
is a homotopy pullback, and this in turn follows from the fact that $[W,-]$ preserves homotopy pullbacks and geometrical discreteness, and that {$\vert-\vert$} preserves homotopy pullbacks along morphisms of geometrically discrete stacks \cite[thm. 3.8.19]{schreiber}. If the pointed stack $(Z,z)$ is geometrically discrete, then so is the stack $\Omega Z$ (pointed at the constant loop at $z$), and so
%. So we see that for a geometrically discrete $Z$,
%
%In what follows we will always assume $Z$ to be geometrically discrete, so we see that,
%as soon as the space of lifts $\mathbf{H}_{/Y}(\tau,\psi)$ is nonempty, any lift $\sigma$ of $\tau$ induces equivalences of $\infty$-groupoids
%%
\begin{equation}
%\Omega_\sigma(\mathbf{H}^\infty_{/\mathscr{B}GL(n)})_{/\tilde{\rho}}(T,\Psi)\cong
%\Omega_\alpha\mathbf{H}_{/Y}(\tau,\psi) \cong
\Omega^2 \mathbf{H}^{\infty}(M, Z)\cong\Omega \mathbf{H}^{\infty}(M,\Omega Z)\cong \mathbf{H}^{\infty}(M,\Omega^2 Z).
\end{equation}
%
%\begin{equation}
%\begin{array}{rl}
%\Omega_{\sigma}(\mathbf{H}^{\infty}_{/\mathscr{B}GL(n)})_{/\tilde{\rho}}((M,\tau),(X,\psi)) & \cong \Omega_{\sigma}(\mathbf{H}^{\infty}(M,\Omega Z))\\
%&\cong \Omega(\Pi[M,\Omega Z])\\
%&\cong \Pi(\Omega([M,\Omega Z]))\\
%&\cong \Pi([M,\Omega^{2} Z])\\
%& \cong \mathbf{H}^{\infty}(M,\Omega^2 Z)
%\end{array}
%\end{equation}
%%
%where we have made explicit the needed compatibility between $\Omega$ and $\Pi$.
Therefore, we can assemble the general considerations of the previous section in the following
\begin{theorem}\label{prop.extension}
 Let $\psi\colon X\to Y$ be the homotopy fibre of a morphisms of smooth $\infty$-stacks $Y\to{Z}$, where $Z$ is pointed and geometrically discrete. For any $\tilde{\rho}$-structured manifold $(M,\tau)$, we have a sequence of natural homotopy pullbacks
\begin{equation}\label{extension}
    \xymatrix{
      \mathbf{H}^{\infty}(M,\Omega^2 Z)\ar[r]
      \ar[d]
      &\mathrm{Diff}^\rho(M,\sigma)
      \ar[d]^{\psi_*}\ar[r]&{*}\ar[d]
            \\
      {*}
      \ar[r]
      & \mathrm{Diff}^{\tilde\rho}(M,\tau)\ar[r]&  \mathbf{H}^{\infty}(M,\Omega Z)
    }
 \end{equation}
 whenever a lift to of $\tau$ to a $\rho$-structure $\sigma$ exists.
\end{theorem}

\subsection{The case of manifolds with boundary}
With an eye to topological quantum field theories, it is interesting to consider also the case of $k$-dimensional manifolds with boundary $(M,\partial M)$. Since the boundary $\partial M$ comes with a collar in $M$, i.e. with a neighbourhood in $M$ diffeomorphic to $\partial M\times [0,1)$ the restriction of the tangent bundle of $M$ to $\partial M$ splits as $TM\vert_{\partial M}\cong T\partial M\oplus \underline{\mathbb{R}}_{\partial M}$ and this gives a natural homotopy commutative diagram
\begin{equation}
\xymatrix{
\partial M\ar[rd]_{T^{\mathrm{st}}}\ar[rr]^{\iota}&& M\ar[dl]^{T^{\mathrm{st}}}\\
&\mathscr{B}GL(n)
}
\end{equation}
for any $n\geq k$. In other words, the embedding of the boundary, $\iota\colon \partial M\to M$ is naturally a morphism in the slice over $\mathscr{B}GL(n)$. This means that any $\tilde{\rho}$-framing on $M$ can be pulled back to a $\tilde{\rho}$-framing on $\partial M$:
\begin{equation}
\iota^*\colon \mathbf{H}^\infty_{/\mathscr{B}GL(n)}(T^{\mathrm{st}},\tilde{\rho})\to \mathbf{H}^\infty_{/\mathscr{B}GL(n)}(T^{\mathrm{st}}\bigr\vert_{\partial M},\tilde{\rho}).
\end{equation}
That is, for any  $\tilde{\rho}$-framing on $M$ we have a natural homotopy commutative diagram
\begin{equation}
\xymatrix{
\partial M\ar[rd]^{\tau\vert_{\partial M}}\ar[rdd]_{T^{\mathrm{st}}\vert_{\partial M} }\ar[rr]^{\iota}&& M\ar[dl]_{\tau}\ar[ddl]^{T^{\mathrm{st}}}\\
&Y\ar[d]\\
&\mathscr{B}GL(n)
}
\end{equation}
realizing $\iota$ as a morphism in the slice over $Y$. Therefore we have a further pullback morphism
\begin{equation}
\iota^*\colon \mathbf{H}^\infty_{/Y}(\tau,\psi)\to \mathbf{H}_{/Y}(\tau\vert_{\partial M},\psi)
\end{equation}
for any morphism $\psi\colon (X,\rho)\to (Y,\tilde{\rho})$ in the slice over $\mathscr{B}GL(n)$. For any fixed $\rho$-framing $\text{\cyr{zh}}$ on $\partial M$ we can then form the space of $\rho$-framings on the $\tilde{\rho}$-framed manifold $M$ extending $\text{\cyr{zh}}$. This is the homotopy fibre of $\iota^*$ at $\text{\cyr{zh}}$:
\begin{equation}
\xymatrix{
\mathbf{H}_{/Y}^{\infty,\text{\cyr{zh}}}((M,\partial M,\tau),(X,\psi))\ar[r]\ar[d]&{*}\ar[d]^{\text{\cyr{zh}}}\\
\mathbf{H}_{/Y}(\tau,\psi)\ar[r]^-{\iota^*} &\mathbf{H}_{/Y}(\tau\vert_{\partial M}),\psi)}
\end{equation}
Reasoning as in Section \ref{homotopy-fibres}, when the morphism $\psi\colon X\to Y$ is the homotopy fibre of a morphism $c\colon Y\to Z$ one sees that, as soon as the $\rho$-structure $\text{\cyr{zh}}$ on $\partial M$ can be extended to a $\rho$-structure on $M$, then the space $\mathbf{H}_{/Y}^{\infty,\text{\cyr{zh}}}((M,\partial M,\tau),(X,\psi))$ of such extensions is a torsor for the $\infty$-group $\mathbf{H}^{\infty,\mathrm{rel}}(M,\partial M;\Omega Z)$ defined by the homotopy pullback
\begin{equation}
\xymatrix{
\mathbf{H}^{\infty,\mathrm{rel}}(M,\partial M;\Omega Z)\ar[r]\ar[d]&{*}\ar[d]^{\mathbf{0}}\\
\mathbf{H}^\infty(M,\Omega Z)\ar[r]^{\iota^*} &\mathbf{H}^\infty(\partial M,\Omega Z)}
\end{equation}
In particular, for $Z=\mathbf{B}^nA$ for some discrete abelian group $A$, the space $\mathbf{H}^{\infty,\mathrm{rel}}(M,\partial M;\mathbf{B}^{n-1}A)$ is the space whose set of connected components is the $(n-1)$-th relative cohomology group of $(M,\partial M)$:
\begin{equation}
\pi_0\mathbf{H}^{\infty,\mathrm{rel}}(M,\partial M;\mathbf{B}^{n-1}A)\cong H^{n-1}(M,\partial M;A).
\end{equation}
Moreover, since $\mathbf{B}^nA$ is $(n-1)$-connected, we see that any homotopy from $c\circ \tau\vert_{\partial M}\colon \partial M \to \mathbf{B}^nA$ to the trivial map can be extended to a homotopy from $c\circ \tau\colon M \to \mathbf{B}^nA$ to the trivial map, as soon as $\dim M<n$. In other words, for $Z=\mathbf{B}^nA$,  if $k<n$ every $\rho$-structure on $\partial M$ can be extended to a $\rho$-structure on $M$.
\\
\\
The space $\mathbf{H}_{/Y}^{\infty,\text{\cyr{zh}}}((M,\partial M,\tau),(X,\psi))$  has a natural interpretation in terms of ${\rho}$-framed cobordism: it is the space of morphisms from the empty manifold to the ${\rho}$-framed  manifold $(\partial M,\text{\cyr{zh}})$, whose underlying non-framed cobordism is $M$. As such, it carries a natural action of the $\infty$-group of ${\rho}$-framings on the cylinder $\partial M\times [0,1]$ which restrict to the ${\rho}$-framing $\text{\cyr{zh}}$ both on $\partial M\times\{0\}$ and on $\partial M\times\{1\}$. These are indeed precisely the $\rho$-framed cobordisms lifting the trivial non-framed cobordism. Geometrically this action is just the glueing of such a $\rho$-framed cylinder along $\partial M$, as a collar in $M$. On the other hand, by the very definition of $\mathbf{H}^\infty$, this $\infty$-group of ${\rho}$-framed cylinders is nothing but the loop space $\Omega_{\text{\cyr{zh}}}(\mathbf{H}^\infty_{/\mathscr{B}GL(n)})_{/{\rho}}(T^{\mathrm{st}}\bigr\vert_{\partial M}, \psi)$, i.e., the loop space at $\text{\cyr{zh}}$ of the space of $\rho$-structures on $\partial M$ lifting the $\tilde{\rho}$-structure $\tau\vert_{\partial M}$. Comparing this to the diagram (\ref{diag2}), we see that the space of $\rho$-structures on $M$ extending a given $\rho$-structure on $\partial M$ comes with a natural action of the $\infty$-group
which is the centre of the extension $\mathrm{Diff}^{\tilde\rho}(\partial M,\text{\cyr{zh}})$ of $\mathrm{Diff}^{\tilde\rho}(M,\tau\vert_{\partial M})$.\footnote{This should be compared to Segal's words in \cite{SegalCFT}: ``An oriented 3-manifold $Y$ whose boundary $\partial Y$ is rigged has itself a set of riggings which form a principal homogeneous set under the group $\mathbb{Z}$ which is the centre of the central extension of $\mathrm{Diff}(\partial Y)$.''}
In the case $\psi\colon X\to Y$ is the homotopy fibre of a morphism $c\colon Y\to \mathbf{B}^nA$, passing to equivalence classes we find the natural action of $H^{n-2}(\partial M,A)$ on the relative cohomology group $H^{n-1}(M,\partial M;A)$ given by the suspension isomorphism $H^{n-2}(\partial M,A)\cong H^{n-1}(\partial M\times [0,1],\partial M\times\{0,1\},A)$ combined with the natural translation action
\begin{equation}
H^{n-1}(M,\partial M;A)\times H^{n-1}(\partial M\times [0,1],\partial M\times\{0,1\},A)\to H^{n-1}(M,\partial M;A).
\end{equation}
For instance, if $M$ is a connected oriented 3-manifold with connected boundary $\partial M$ and we choose $n=4$ and $A=\mathbb{Z}$, then we get the translation action of $\mathbb{Z}$ on itself.\footnote{Again, compare to Segal's prescription on the set of riggings on a oriented 3-manifold.}

 \section{Mapping class groups of $\rho$-framed manifolds}\label{section.mapping}
In this final section, we consider an application of the general notion of $\rho$-structure developed in the previous sections to investigate extensions of the mapping class group of smooth manifolds.\\
Inspired by the classical notion of mapping class group, see for instance \cite{HatcherDiffeo}, we consider the following
\begin{definition}
Let $M$ be a $k$-dimensional manifold, and let $\rho\colon X\to  \mathscr{B}GL(n)$ be a morphisms of smooth $\infty$-stacks, with $k\leq n$.
The \emph{mapping class group} $\Gamma^\rho(M,\sigma)$ of a $\rho$-framed manifold $(M,\sigma)$ is the group of connected components of the $\rho$-diffeomorphism $\infty$-group of $(M,\sigma)$, namely
\begin{equation}
\Gamma^\rho(M,\sigma):=\pi_0\mathrm{Diff}^\rho(M,\sigma)
\end{equation}
\end{definition}
In the setting of the Section \ref{section.lifting}, we consider the case in which the $\infty$-stack $X$ is the homotopy fiber of a morphism $Y\to{Z}$, with $Z$ a geometrically discrete $\infty$-stack.
Then, induced by diagram (\ref{extension}), we have the following long exact sequence in homotopy
\begin{equation}\label{exactsequence}
\cdots \to\pi_1\mathrm{Diff}^\rho(M,\sigma) \to\pi_1\mathrm{Diff}^{\tilde\rho}(M,\tau)\to \pi_2\mathbf{H}^{\infty}(M,Z)\to \Gamma^\rho(M,\sigma)\to \Gamma^{\tilde{\rho}}(M,\tau)\to \pi_1\mathbf{H}^{\infty}(M,Z).
\end{equation}
Notice that the morphism
\begin{equation}
\Gamma^{\tilde{\rho}}(M,\tau)\to \pi_1\mathbf{H}^{\infty}(M,Z)
\end{equation}
is a homomorphism at the $\pi_0$ level, so it is only a morphism of pointed sets and \emph{not} a morphism of groups. It is the morphism that associates with a $\rho$-diffeomorphism $f$ the pullback of the lift $\sigma$ of $\tau$. In other words, it is the morphism of pointed sets from the set of isotopy classes of $\rho$-diffeomorphisms to the set of equivalence classes of lifts induced by the natural action
\begin{equation}
\Gamma^{\tilde{\rho}}(M,\tau) \times \{\text{(equivalence classes of) lifts of }\tau\}\to \{\text{(equivalence classes of) lifts of }\tau\}
\end{equation}
once one picks a distinguished element $\sigma$ in the set (of equivalence classes of) of lifts and uses it to identify this set with $\pi_0\mathbf{H}^{\infty}(M,\Omega Z)\cong \pi_1\mathbf{H}^{\infty}(M,Z)$.
A particularly interesting situation is the case when $c$ is a degree $d$ characteristic class for $Y$, i.e., when $c\colon Y\to \mathbf{B}^{d}A$ for some discrete abelian group $A$, and $M$ is a closed manifold. Since $\mathbf{B}^{d}A$ is a geometrically discrete $\infty$-stack, we have that $\mathbf{H}^{\infty}(M,\mathbf{B}^{d}A)$ is equivalent, as an $\infty$-groupoid, to
$\mathbf{H}(M,\mathbf{B}^{d} A)$ . Consequently, we obtain that $\pi_k\mathbf{H}^{\infty}(M,\mathbf{B}^{d}A)=H^{d-k}(M,A)$ for $0\leq k\leq d$ (and zero otherwise): in particular, the obstruction to lifting a
 $\tilde{\rho}$-framing $\tau$ on $M$  to a $\rho$-framing $\sigma$ is given by an element in $H^{d}(M,A)$. When this obstruction vanishes, hence when a lift $\sigma$ of $\tau$ does exist, the long exact sequence above reads as
\begin{equation}\label{seq1}
\cdots \to\pi_1\mathrm{Diff}^{\tilde\rho}(M,\tau)\to H^{d-2}(M,A)\to \Gamma^\rho(M,\sigma)\to \Gamma^{\tilde{\rho}}(M,\tau)\to H^{d-1}(M,A)
\end{equation}
for $d\geq 2$, and simply as
\begin{equation}\label{seq2}
\cdots \to\pi_1\mathrm{Diff}^{\tilde\rho}(M,\tau)\to 1\to \Gamma^\rho(M,\sigma)\to \Gamma^{\tilde{\rho}}(M,\tau)\to H^{0}(M,A)
\end{equation}
for $d=1$.
\begin{remark}
The long exact sequences (\ref{seq1}) and (\ref{seq2}) are a shadow of Theorem \ref{prop.extension}, which is a more general extension result for the \emph{whole} $\infty$-group $\mathrm{Diff}^\rho(M,\sigma)$.
\end{remark}
The morphism of pointed sets $\Gamma^{\tilde{\rho}}(M,\tau)\to H^{d-1}(M,A)$ is easily described: once a lift $\sigma$ for $\tau$ has been chosen, the space of lifts is identified with $\mathbf{H}^{\infty}(M,\mathbf{B}^{d-1}A)$ and the natural pullback action of the $\tilde{\rho}$-diffeomorphism group of $M$ on the space of maps from $M$ to $\mathbf{B}^{d-1}A$ induces the morphism
\begin{equation}
\begin{array}{rl}
\mathrm{Diff}^{\tilde\rho}(M,\tau)&\to \mathbf{H}^{\infty}(M,\mathbf{B}^{d-1}A)\\
f&\mapsto f^*\sigma-\sigma
\end{array}
\end{equation}
where we have written $f^*\sigma-\sigma$ for the element in $\mathbf{H}^{\infty}(M,\mathbf{B}^{d-1}A)$ which represents the ``difference'' between $f^*\sigma$ and $\sigma$ in the space of lifts of $\tau$ seen as a $\mathbf{H}^{\infty}(M,\mathbf{B}^{d-1}A)$-torsor. The morphism $\Gamma^{\tilde{\rho}}(M,\tau)\to H^{d-1}(M,A)$ is obtained by passing to $\pi_0$'s and so we see in particular from the long exact sequence (\ref{seq1}) that the image of $\Gamma^{\rho}(M,\tau)$ into $\Gamma^{\tilde{\rho}}(M,\tau)$ consist of precisely the isotopy classes of those $\tilde{\rho}$-diffeomorphisms of $(M,\tilde{\rho})$ which fix the $\rho$-structure $\sigma$ up to homotopy.
\\
Similarly, for $d\geq 2$, the morphism of groups $\pi_1\mathrm{Diff}^{\tilde\rho}(M,\tau)\to H^{d-2}(M,A)$ in sequence (\ref{seq1}) can be described explicitly as follows. A closed path $\gamma$ based at the identity in $\mathrm{Diff}^{\tilde\rho}(M,\tau)$ defines then a morphism $\gamma^\#\colon M\times [0,1]\to \mathbf{B}^{d-1}A$, as the composition
\begin{equation}
M\times [0,1]\to M\xrightarrow{\mathbf{0}} \mathbf{B}^{d-1}A,
\end{equation}
where the first arrow is the homotopy from the identity of $M$ to itself and where $\mathbf{0}\colon M\to \mathbf{B}^{d-1}A$ is the \emph{collapsing} morphism, namely the morphism obtained as the composition $M\to{*}\to{\mathbf{B}^{d-1}A}$ (here we are using that $\mathbf{B}^{d-1}A$ comes naturally equipped with a base point). The image of $[\gamma]$ in $H^{d-2}(M,A)$ is then given by the element $[\gamma^\#]$ in the relative cohomology group
\begin{equation}
H^{d-1}(M\times[0,1],M\times\{0,1\},A)\cong H^{d-1}(\mathbf{\Sigma} M,A)\cong H^{d-2}(M,A)\,.
\end{equation}
By construction, $[\gamma^\#]$  is the image in $H^{d-1}(M\times[0,1],M\times\{0,1\},A)\cong H^{d-2}(M,A)$ of the zero class in $H^{d-1}(M,A)$ via the pullback morphism $M\times [0,1]\to M$, so it is the zero class in $H^{d-1}(M\times[0,1],M\times\{0,1\},A)$. That is, the morphism $\pi_1\mathrm{Diff}^{\tilde\rho}(M,\tau)\to H^{d-2}(M,A)$ is the zero morphism, and we obtain the short exact sequence
 \begin{equation}\label{extension.h}
1\to H^{d-2}(M,A)\to \Gamma^\rho(M,\sigma)\to \Gamma^{\tilde{\rho}}(M,\tau) \to H^{d-1}(M,A)
\end{equation}
showing that $\Gamma^\rho(M,\sigma)$ is a $H^{d-2}(M,A)$-extension of a subgroup of $\Gamma^{\tilde{\rho}}(M,\tau)$: namely, the subgroup is the $\Gamma^{\tilde{\rho}}(M,\tau)$-stabilizer of the element of $H^{d-1}(M,A)$ corresponding to the lift $\sigma$ of $\tau$. The action of this stabiliser on $H^{d-2}(M,A)$ is the pullback action of $\tilde{\rho}$-diffeomorphisms of $M$ on the $(d-2)$-th cohomology group of $M$ with coefficients in $A$. Since this action is not necessarily trivial, the $H^{d-2}(M,A)$-extension $\Gamma^\rho(M,\sigma)$ of the stabiliser of $\sigma$ is not a central extension in general.
\subsection{Oriented and spin manifolds, and $r$-spin surfaces}
Before discussing $p_1$-structures and their modular groups, which is the main goal of this note, let us consider two simpler but instructive examples: oriented manifolds and spin curves.\\

Since the $\infty$-stack $\mathscr{B}SO(n)$ is the homotopy fibre of the first Stiefel-Whitney class
\begin{equation}
w_1\colon \mathscr{B}O(n)\to \mathbf{B}\mathbb{Z}/2\mathbb{Z}
\end{equation}
an $n$-dimensional manifold can be oriented if and only if $[w_1\circ T_M]$ is the trivial element in $\pi_0\mathbf{H}^{\infty}(M,\mathbf{B}\mathbb{Z}/2\mathbb{Z})=H^1(M,\mathbb{Z}/2\mathbb{Z})$. When this happens, the space of possible orientations on $M$ is equivalent to $\mathbf{H}^{\infty}(M,\mathbb{Z}/2\mathbb{Z})$, so when $M$ is connected it is equivalent to a 2-point set. For a fixed orientation on $M$, we obtain from (\ref{seq2}) with $A=\mathbb{Z}/2\mathbb{Z}$ the exact sequence
\begin{equation}
1\to \Gamma^{\mathrm{or}}(M)\to \Gamma(M) \to \mathbb{Z}/2\mathbb{Z}
\end{equation}
where $\Gamma^{\mathrm{or}}(M)$ denotes the mapping class group of oriented diffeomorphisms of $M$, and where the rightmost morphism is induced by the action of the diffeomorphism group of $M$ on the set of its orientations. The oriented mapping class group of $M$ is therefore seen to be a subgroup of order 2 in $\Gamma(M)$ in case there exists at least an orientation reversing diffeomorphism of $M$, and to be the whole $\Gamma(M)$ when such a orientation reversing diffeomorphism does not exist (e.g., for $M=\mathbb{P}^{n/2}\mathbb{C}$, for $n\equiv 0 \mod 4$).\\

Consider now the $\infty$-stack $\mathscr{B}\mathrm{Spin}(n)$ for $n\geq 3$. It can be realised as the homotopy fibre of the second Stiefel-Whitney class
\begin{equation}
w_2\colon \mathscr{B}SO(n)\to \mathbf{B}^2\mathbb{Z}/2\mathbb{Z}.
\end{equation}
An oriented $n$-dimensional manifold $M$ will then admit a spin structure if and only if $[w_2\circ T_M]$ is the trivial element in $\pi_0\mathbf{H}^{\infty}(M,\mathbf{B}^2\mathbb{Z}/2\mathbb{Z})=H^2(M,\mathbb{Z}/2\mathbb{Z})$. When this happens, the space of possible orientations on $M$ is equivalent to $\mathbf{H}^{\infty}(M,\mathbf{B}\mathbb{Z}/2\mathbb{Z})$, and we obtain, for a given spin structure $\sigma$ on $M$ lifting the orientation of $M$, the exact sequence
\begin{equation}
1\to H^0(M,\mathbb{Z}/2\mathbb{Z})\to \Gamma^{\mathrm{Spin}}(M,\sigma)\to \Gamma^{\mathrm{or}}(M)\to H^1(M,\mathbb{Z}/2\mathbb{Z}).
\end{equation}
In particular, if $M$ is connected, we get the exact sequence
\begin{equation}
1\to \mathbb{Z}/2\mathbb{Z}\to \Gamma^{\mathrm{Spin}}(M,\sigma)\to \Gamma^{\mathrm{or}}(M)\to H^1(M,\mathbb{Z}/2\mathbb{Z}).
\end{equation}
Since, for a connected $M$, the pullback action of oriented diffeomorphisms on $H^0(M,\mathbb{Z}/2\mathbb{Z})$ is trivial, we see that in this case the group $\Gamma^{\mathrm{Spin}}(M,\sigma)$ is a $\mathbb{Z}/2\mathbb{Z}$-central extension of the subgroup of $\Gamma^{\mathrm{or}}(M)$ consisting of (isotopy classes of) orientation preserving diffeomorphisms of $M$ which fix the spin structure $\sigma$ (up to homotopy). The group $\Gamma^{\mathrm{Spin}}(M,\sigma)$ and its relevance to Spin TQFTs are discussed in detail in \cite{masbaum}.\\

For $n=2$, the homotopy fibre of $w_2\colon \mathscr{B}SO(2)\to \mathbf{B}^2\mathbb{Z}/2\mathbb{Z}$ is again $\mathscr{B}SO(2)$ with the morphism $\mathscr{B}SO(2)\to \mathscr{B}SO(2)$ induced by the group homomorphism
\begin{equation}
\begin{array}{rl}
SO(2)&\to SO(2)\\
x&\mapsto x^2
\end{array}
\end{equation}
Since the second Stiefel-Withney class of an oriented surface $M$ is the $\textrm{mod 2}$ reduction of the first Chern class of the holomorphic tangent bundle of $M$ (for any choice of a complex structure compatible with the orientation), and $\langle c_1(T^{\mathrm{hol}})M| [M]\rangle =2-2g$, where $g$ is the genus of $M$, one has that $[w_2\circ T_M]$ is always the zero element in $H^2(M,\mathbb{Z}/2\mathbb{Z})$ for a compact oriented surface, and so the orientation of $M$ can always be lifted to a spin structure. More generally, one can consider the group homomorphism $SO(2)\to SO(2)$ given by $x\mapsto x^r$, with $r\in \mathbb{Z}$. We have then a homotopy fibre sequence
\begin{equation}
\xymatrix{
\mathscr{B}SO(2)\ar[rr]\ar[d]_{\rho_{1/r}}&&{*}\ar[d]\\
\mathscr{B}SO(2)\ar[rr]^{c_{(x\to x^r)}}&&\mathbf{B}^2\mathbb{Z}/2\mathbb{Z}
}
\end{equation}
In this case one sees that an \emph{$r$-spin structure} on an oriented surface $M$, i.e. a lift of the orientation of $M$ through $\rho_{1/r}$, exists if and only if $2-2g\equiv 0 \mod r$. When this happens, one obtains the exact sequence
\begin{equation}
1\to \mathbb{Z}/r\mathbb{Z}\to \Gamma^{1/r}(M,\sigma)\to \Gamma^{\mathrm{or}}(M)\to H^1(M,\mathbb{Z}/r\mathbb{Z}),
\end{equation}
which exhibits the $r$-spin mapping class group $\Gamma^{1/r}(M,\sigma)$ as a  $\mathbb{Z}/r\mathbb{Z}$-central extension of the subgroup of $\Gamma^{\mathrm{or}}(M)$ consisting of isotopy classes of orientation preserving diffeomorphisms of $M$  fixing the $r$-spin structure $\sigma$ (up to homotopy). The group $\Gamma^{1/r}(M,\sigma)$ appears as the fundamental group of the moduli space of $r$-spin Riemann surfaces, see \cite{randal-williams1,randal-williams2}.

\subsection{$p_1$-structures on oriented surfaces}
\label{sect.p1structures}
Let us now finally specialise the general construction above to the case of $p_1$-structures on closed oriented surfaces, to obtain the $\mathbb{Z}$-central extensions considered in \cite{SegalCFT} around page 476. In particular we will see, how $p_1$-structures provide a simple realisation of Segal's idea of extended surfaces and 3-manifolds (see also \cite{BunkeNaumann, CHMV}).\footnote{In \cite{SegalCFT}, the extension is defined in terms of ``riggings'', a somehow ad hoc construction depending on the contractiblity of Teichm\"fuller spaces and on properties of the $\eta$-invariant of metrics on 3-manifolds. Segal says: ``I've not been able to think of a less sophisticated definition of a rigged surface, although there are many possible variants. The essential idea is to associate \emph{functorially} to a smooth surface a space -such as $\mathcal{P}_X$- which has fundamental group $\mathbb{Z}$.''}To this aim, our stack $Y$ will be the stack $\mathscr{B}SO(n)$ for some $n\geq 3$, the stack $Z$ will be $\mathbf{B}^4\mathbb{Z}$ and the morphism $c$ will be the first Pontryagin class $p_1\colon \mathscr{B}SO(n)\to \mathbf{B}^4\mathbb{Z}$. the stack $X$ will be the homotopy fiber of $p_1$, and so the morphism $\psi$ will be the morphism
\begin{equation}
\rho_{p_1}\colon \mathbf{B}^3\mathbb{Z}/\!/_hSO(n) \to \mathscr{B}SO(n).
\end{equation}
of example \ref{charclass}. A lift $\sigma$ of an orientation on a manifold $M$ of dimension at most 3 to a morphism $M\to \mathbf{B}^3\mathbb{Z}/\!/_hSO(n)$ over $\mathscr{B}O(n)$ will be called a $p_1$-struture on $M$. That is, a pair $(M,\sigma)$ is the datum of a smooth oriented manifold $M$ together with a trivialisation of its first Pontryagin class. Note that, since $p_1$ is a degree four cohomology class, it can always be trivialised on manifolds of dimension at most 3. In particular, when $M$ is a closed connected oriented 3-manifold, we see that the space of lifts of the orientation of $M$ to a $p_1$ structure, is equivalent to the space $\mathbf{H}(M, \mathbf{B}^3\mathbb{Z})$ and so its set of connected components is
\begin{equation}
\pi_0\mathbf{H}(M,\mathbf{B}^3\mathbb{Z})=H^3(M,\mathbb{Z})\cong \mathbb{Z}.
\end{equation}
In other words, there is a  $\mathbb{Z}$-torsor of equivalence classes of $p_1$-strctures on a connected oriented 3-manifold. Similarly, in the relative case, i.e., when $M$ is a connected oriented 3-manifold with boundary, the set of equivalence classes of $p_1$-strctures on $M$ extending a given $p_1$-structure on $\partial M$ is nonempty and is a torsor for the
relative cohomology group
\begin{equation}
H^3(M,\partial M;\mathbb{Z})\cong \mathbb{Z},
\end{equation}
in perfect agreement with the prescription in \cite[page 480]{SegalCFT}.\footnote{The naturality of the appearance of this $\mathbb{Z}$-torsor here should be compared to Segal's words in \cite{SegalCFT}: ``An oriented 3-manifold $Y$ whose boundary $\partial Y$ is rigged has itself a set of riggings which form a principal homogeneous set under the group $\mathbb{Z}$ which is the centre of the central extension of $\mathrm{Diff}(\partial Y)$. I do not know an altogether straightforward way to define a rigging of a 3-manifold.'' Rigged 3-manifolds are then introduces by Segal in terms of the space of metrics on the 3-manifold $Y$ and of the $\eta$-invariant of these metrics.}\\
We can now combine the results of the previous section in the following
\begin{proposition}
Let $M$ be a connected oriented surface, and let $\sigma$ be a $p_{1}$-structure on $M$. We have then the following central extension
\begin{equation}
1\to \mathbb{Z}\to \Gamma^{p_1}(M,\sigma)\to \Gamma^{\mathrm{or}}(M) \to 1,
\end{equation}
where $\Gamma^{p_1}$ as a shorthand notation for $\Gamma^{\rho_{p_1}}$.
\end{proposition}
\begin{proof}
Since $M$ is oriented, we have a canonical isomorphism $H^2(M,\mathbb{Z})\cong \mathbb{Z}$ induced by Poincar\'e duality. Moreover, since $M$ is connected, from \ref{extension.h} we obtaine the following short exact sequence
\begin{equation}
1\to \mathbb{Z}\to \Gamma^{\rho_{p_1}}(M,\sigma)\to \Gamma^{\mathrm{or}}(M) \to 1
\end{equation}
Finally, since the oriented diffeomorphisms action on $H^2(M,\mathbb{Z})$ is trivial for a connected oriented surface $M$, this short exact sequence is a $\mathbb{Z}$-central extension.
\end{proof}
%%
%In the case $M$ is a connected oriented surface, the results of the previous section, together with the canonical Poincar\'e duality isomorphism $H^2(M,\mathbb{Z})\cong \mathbb{Z}$, we get, for any fixed $p_1$-structure $\sigma$ on $M$, a short exact sequence
%of groups
% \begin{equation}
%1\to \mathbb{Z}\to \Gamma^{p_1}(M,\sigma)\to \Gamma^{\mathrm{or}}(M) \to 1,
%\end{equation}
%where we used $\Gamma^{p_1}$ as a shorthand notation for $\Gamma^{\rho_{p_1}}$. Since the oriented diffeomorphisms action on $H^2(M,\mathbb{Z})$ is trivial for a connected oriented surface $M$, this short exact sequence is a $\mathbb{Z}$-central extension.
%
%\section{Outlook}
%%{Say something about 3d TQFTs, and about higher extensions, e.g., one can consider Spin manifolds and %$\frac{1}{2}p_1$-structures or, even better, String manifolds  and $\frac{1}{6}p_2$-strcutures. All this to be investigated %elsewhere. Also, and much needed, we have to say here or in the Introduction that a version of this note, dealing with more general %higher extensions and with fully detailed and rigorous proofs is in preparation.}
%
%One is interested in higher dimensional Chern-Simons theories, the notable next case being
%7-dimensional Chern-Simons theory. The above discussion gives a general means to determine
%and construct the relevant higher extensions of mapping class groups of higher dimensional manifolds.
%More on this is going to be discussed elsewhere.
￼

\section*{Appendix: proof of the extension theorem}
%\addcontentsline{toc}{section}{Appendix: proof of the extension theorem}
Here we provide the details for proof of the existence of the homotopy fibre sequence (\ref{diag1}), which is the extension theorem this note revolves around. All the notations in this Appendix are taken from Section \ref{section.extensions}.
\begin{lemmaapp}
We have a  homotopy pullback diagram
\begin{equation}
 \xymatrix{
      \mathrm{Diff}^\rho(M,\Sigma)
      \ar[r]
      \ar[d]_{\psi_*}
      &
      \mathbf{Aut}^\infty_{/\rho}(\sigma)
      \ar[d]^{\psi_*}\\
       \mathrm{Diff}^{\tilde{\rho}}(M,T)
      \ar[r]
      &
      \mathbf{Aut}^\infty_{/\tilde{\rho}}(\tau)
  }
\end{equation}
\end{lemmaapp}
\begin{proof}
By definition of  (equation (\ref{def.diff})), we have homotopy pullback diagrams
\begin{equation}
    \xymatrix{
      \mathrm{Diff}^\rho(M,\Sigma)
      \ar[r]
      \ar[d]
      &
      \mathbf{Aut}^\infty_{/\rho}(\sigma)
      \ar[d]
      \\
      \mathrm{Diff}(M)
      \ar[r]
      &
      \mathbf{Aut}^\infty_{/\mathscr{B}GL(n)}(T^{\mathrm{st}}_M)
    }
 \end{equation}
and
\begin{equation}
    \xymatrix{
      \mathrm{Diff}^{\tilde{\rho}}(M,T)
      \ar[r]
      \ar[d]
      &
      \mathbf{Aut}^\infty_{/\tilde{\rho}}(\tau)
      \ar[d]
      \\
      \mathrm{Diff}(M)
      \ar[r]
      &
      \mathbf{Aut}^\infty_{/\mathscr{B}GL(n)}(T^{\mathrm{st}}_M)
    }
 \end{equation}
By pasting them together as
\begin{equation}
    \xymatrix{
      \mathrm{Diff}^\rho(M,\Sigma)
      \ar[r]
      \ar[d]_{\psi_*}
      &
      \mathbf{Aut}^\infty_{/\rho}(\sigma)
      \ar[d]^{\psi_*}\\
       \mathrm{Diff}^{\tilde{\rho}}(M,T)
      \ar[r]
      \ar[d]
      &
      \mathbf{Aut}^\infty_{/\tilde{\rho}}(\tau)
      \ar[d]
      \\
      \mathrm{Diff}(M)
      \ar[r]
      &
      \mathbf{Aut}^\infty_{/\mathscr{B}GL(n)}(T^{\mathrm{st}}_M)
    }
 \end{equation}
and by the 2-out-of-3 law for homotopy pullbacks the claim follows.
\end{proof}
We need the following basic fact \cite[Lemma 5.5.5.12]{lurieHTT}:
\begin{lemmaapp}\label{lurieLemma}
  Let $\mathbf{C}$ be an $\infty$-category, $\mathbf{C}_{/x}$ its slice over
  an object $x \in \mathbf{C}$, and let $f \colon a \to x$ and $g \colon b \to x$
  be two morphisms into $x$. Then
  the hom space $\mathbf{C}_{/x}(f,g)$ in the slice
  is expressed in terms of that in $\mathbf{C}$ by the fact that there
  is a homotopy pullback (in $\infty \mathrm{Grpd}$) of the form
  $$
    \xymatrix{
       \mathbf{C}_{/x}(f,g)
       \ar[r]
       \ar[d]
       &
       \mathbf{C}(a,b)
       \ar[d]^{g \circ (-)}
       \\
       \ast
       \ar[r]^{[f]}
       &
       \mathbf{C}(a,x)
    }
  $$
  where the right morphism is composition with $g$, and where
  the bottom morphism picks $f$ regarded as a point in $\mathbf{C}(a,x)$.
\end{lemmaapp}
\begin{lemmaapp}
We have homotopy pullback diagrams
\begin{equation}
\raisebox{20pt}{
\xymatrix{
\Omega_{T}\mathbf{H}^\infty_{/\mathscr{B}GL(n)}(T^{\mathrm{st}}_M,\tilde{\rho})
      \ar[d]
      \ar[r]
   &
      \mathbf{Aut}^{\infty}_{/\tilde{\rho}}(T)\ar[d]\\
{*}\ar[r]& \mathbf{Aut}^\infty_{/\mathscr{B}GL(n)}(T^{\mathrm{st}}_M)
}}\qquad \text{ and }\qquad
\raisebox{20pt}{
\xymatrix{
\Omega_{\Sigma}\mathbf{H}^\infty_{/\mathscr{B}GL(n)}(T^{\mathrm{st}}_M,\rho)
      \ar[d]
      \ar[r]
   &
      \mathbf{Aut}^{\infty}_{\rho}(\Sigma)\ar[d]\\
{*}\ar[r]& \mathbf{Aut}^\infty_{/\mathscr{B}GL(n)}(T^{\mathrm{st}}_M)
}}
\end{equation}
\end{lemmaapp}
\begin{proof}
Let $\mathbf{C}$ be an $(\infty,1)$-category, and let $f\colon x\to y$ be a morphism in $\mathbf{C}$. Then
by Lemma A.\ref{lurieLemma} and using 2-out-of-3 for homotopy pullbacks,
the forgetful morphism $\mathbf{C}_{/y}\to \mathbf{C}$ from the slice over $y$ to $\mathbf{C}$ induces a morphism of $\infty$-groups $\mathbf{Aut}_{\mathbf{C}_{/y}}(f)\to \mathbf{Aut}_{\mathbf{C}}(x)$ sitting in a pasting of homotopy pullbacks
like this:
\begin{equation}
\xymatrix{
  \Omega_f\mathbf{C}(x,y)\ar[r]\ar[d]
  &
  \mathbf{Aut}_{\mathbf{C}_{/y}}(f)\ar[d]
  \ar[r]
  &
  \ast
  \ar[d]^-{[f]}
  \\
  {*}\ar[r]^-{[\mathrm{id}]}
  \ar@/_1pc/[rr]_-{[f]}
  &
  \mathbf{Aut}_{\mathbf{C}}(x)
  \ar[r]^{f\circ (-)}
  &
  \mathbf{C}(x,y)
}
\end{equation}
By taking here $\mathbf{C}=\mathbf{H}^\infty_{/\mathscr{B}GL(n)}$, $x=T^{\mathrm{st}}_M$, $y=\tilde{\rho}$ (resp., $y=\rho$), and $f=T$ (resp., $f=\Sigma$), the left square yields the first (resp., the second) diagram in the statement of the lemma.
\end{proof}

\begin{lemmaapp}
We have a homotopy pullback diagram
\begin{equation}
    \xymatrix{
     \Omega_{\beta}(\mathbf{H}^\infty_{/\mathscr{B}GL(n)})_{/\tilde{\rho}}(T,\Psi)\ar[r] \ar[d]&
      \Omega_{\Sigma}\mathbf{H}^\infty_{/\mathscr{B}GL(n)}(T^{\mathrm{st}}_M,\rho)
            \ar[d]
            \\
   {*}\ar[r] &  \Omega_{T}\mathbf{H}^\infty_{/\mathscr{B}GL(n)}(T^{\mathrm{st}}_M,\tilde{\rho})
      }
\end{equation}
      \end{lemmaapp}
\begin{proof}
%Let $\mathbf{C}$ be an $(\infty,1)$-category and let $f\colon x\to y$ be a morphism in $\mathbf{C}$. The homotopy fiber of
%the push forward morphism
%\begin{equation}
%f_*\colon \mathbf{C}(a,x) \to \mathbf{C}(a,y)
%\end{equation}
%over a given morphisms $\alpha\colon a\to y$ in $\mathbf{C}$ is very easy to describe: it is  the space of lifts of $\alpha$ to a morphism $\beta\colon a \to x$, i.e., it is the space
%$(\mathbf{C}/y)(\alpha,f)$; see, e.g., \cite[Proposition 5.5.5.12]{lurieHTT}. Therefore,
If we take $\mathbf{C}=\mathbf{H}^\infty_{/\mathscr{B}GL(n)}$, $g=(\psi,\Psi)$, $a=T^{\mathrm{st}}_M$, $f=T$, $b=\rho$ and $x=\tilde{\rho}$ in Lemma A.\ref{lurieLemma}, we find the homotopy fibre sequence
\begin{equation}
\xymatrix{
(\mathbf{H}^\infty_{/\mathscr{B}GL(n)})_{/\tilde{\rho}}(T,\Psi)\ar[d]\ar[r]&
 \mathbf{H}^\infty_{/\mathscr{B}GL(n)}(T^{\mathrm{st}},\rho)\ar[d]^{\psi_*}\\
 {*}\ar[r] & \mathbf{H}^\infty_{/\mathscr{B}GL(n)}(T^{\mathrm{st}},\tilde{\rho})}
\end{equation}
By looping the above diagram, the claim follows.
\end{proof}
\begin{lemmaapp}
We have an equivalence of $(\infty,1)$-categories
\begin{equation}
(\mathbf{H}^\infty_{/\mathscr{B}GL(n)})_{/\tilde{\rho}}\cong \mathbf{H}^\infty_{/Y}.
\end{equation}
\end{lemmaapp}
\begin{proof}
Let $\mathbf{C}$ be an $(\infty,1)$-category, and let $f:b\to{x}$ be a 1-morphism in $\mathbf{C}$. By abuse of notation, we can regard $f$ as a diagram $f:\Delta^{1}\to\mathbf{C}$. We have then a morphism
\begin{equation}
\varphi:(\mathbf{C}_{/x})_{/f} \to \mathbf{C}_{/b}
\end{equation}
induced by the $\infty$-functor $\Delta^{0}\hookrightarrow\Delta^{1}$ induced by sending 0 to 1. Since 1 is an initial object in $\Delta^{1}$, the opposite $\infty$-functor is a cofinal map. By noticing that $\mathbf{C}^{op}_{x/}$ is canonically equivalent to $\mathbf{C}_{/x}$, then by \cite[Proposition 4.1.1.8]{lurieHTT} we have that $\varphi$ is an equivalence of $\infty$-categories.
Therefore, if we take $\mathbf{C}=\mathbf{H}^\infty$, and $f=\tilde{\rho}:Y\to\mathscr{B}GL(n)$, we have that the claim follows.
\end{proof}

%%%%%%%%%%%%%%%%%%%%%%%%%%%%%%%%%%%%%%%%%%%%%%%%%%%

\end{document}